\documentclass[10pt]{amsart}
\usepackage{amsmath,amsfonts,latexsym,amssymb,amscd,
  graphicx,amsthm, mathtools,tikz,subcaption,enumerate,hyperref}
\usepackage{ulem,cancel,tikz-cd}
\allowdisplaybreaks[1]

\newcommand{\disteq}{\mathop{{\stackrel{\mathrm{d}}{=}}}}

\newcommand{\ONE}{{\mathbf{1}}}

\newcommand{\N}{{\mathbb N}}

\newcommand{\Ac}{{\mathcal A}}

\newcommand{\Cc}{\mathcal{C}}
\newcommand{\Dc}{\mathcal{D}}
\newcommand{\Fc}{\mathcal{F}}
\newcommand{\Nc}{\mathcal{N}}

\newcommand{\Pp}{\mathsf{P}}
\newcommand{\Z}{{\mathbb Z}}
\newcommand{\iid}{i.i.d.\ }

\newcommand{\E}{\mathsf{E}}

\newcommand{\R}{{\mathbb R}}

\newcommand{\BB}{{\mathbf B}}

\newcommand{\WW}{{\mathbf W}}
\newcommand{\XX}{{\mathbf X}}
\newcommand{\Ss}{{\mathbf S}}

\newcommand{\yy}{{\mathbf y}}

\newcommand{\xx}{{\mathbf x}}

\newcommand{\eps}{{\varepsilon}}

\newtheorem{theorem}{Theorem}[section]

\newtheorem{lemma}{Lemma}[section]

\numberwithin{equation}{section}

\renewenvironment{proof}[1][Proof]{
{\noindent {\sc #1: }}
}{
{{\hfill $\Box$ \smallskip}}
}

\makeatletter
\let\orgdescriptionlabel\descriptionlabel
\renewcommand*{\descriptionlabel}[1]{%
  \let\orglabel\label
  \let\label\@gobble
  \phantomsection
  \edef\@currentlabel{#1}%
  \let\label\orglabel
  \orgdescriptionlabel{#1}%
}
\makeatother

\title[Monotone Map]{On asymptotic behavior of iterates of piecewise constant monotone maps}
\author{Konstantin Khanin}
\author{Liying Li}
\address{Department of Mathematics, University of Toronto, Bahen Center, 40 St.\ George St, Toronto, ON M5S 2E4, Canada }
\email{khanin@math.toronto.edu, liyingli@math.toronto.edu}
\begin{document}

\begin{abstract}
In this paper we study the rate of convergence of the iterates of \iid random
piecewise constant monotone maps to the time-$1$ transport map for the process
of coalescing Brownian motions. We prove that the rate of convergence is given
by a power law. The time-1 map for the coalescing Brownian motions can be viewed
as a fixed point for a natural renormalization transformation acting in the
space of probability laws for random piecewise constant monotone maps.
Our result implies that this fixed point is exponentially stable.
\end{abstract}

\maketitle

\section{ Introduction}

The process of coalescing Brownian motions introduced by Arratia (\cite{ArratiaCBM}) can be described in terms of
piecewise constant monotone maps. It was proved by Arratia that after arbitrary small positive time $t$
the continuum of initial Brownian motions coalesce into a discrete random set of
positions $\{x_i(t), i\in \Z\}$. Moreover, for every $x_i(t)$ there exists an open interval $I_i=(a_i,a_{i+1})$
of initial positions such that, for any $a\in I_i$ the Brownian motion starting at $a$ will reach $x_i$
at time $t$. These intervals do not intersect, and the union of their closures covers the whole~$\R^1$.
Two correlated random fields $\{x_i, i\in \Z\}$ and $\{a_i,i\in \Z\}$ define a random piecewise constant
monotone map $\Phi_{t,0}:\R \to \R$ which maps each $I_i$ into $x_i$. the resulting map can be viewed
as a transport map for the time interval $[0,t]$.

In this paper we shall consider
positive integer values of $t=n$. In the same way one can define transport maps $\Phi_{i+1,i}$.
Then $\Phi_{n,0}=\Phi_{n,n-1}\circ \Phi_{n-1,n-2} \circ \cdots \circ \Phi_{1,0}$. Notice that maps
$\Phi_{i+1,i}$ are independent. Due to exact scaling invariance of Brownian motions, the map $\Phi_{n,0}$
has the same distribution as $\Phi_{1,0}$ up to rescaling by the diffusive factor $\sqrt{n}$.
It means that the distribution of the map $\Phi_{1,0}$ can be viewed as the fixed point for
renormalization transformation corresponding to taking iterates of independent identically distributed
copies of a random time-1 monotone piecewise constant maps, and then rescaling the resulting map.
One can apply such a renormalization transformation to random monotone piecewise constant maps 
$\Psi_{i+1,i}$ with various distributions.

The natural question is whether such process will converge 
to the above fixed
point corresponding to the coalescing Brownian motions. The affirmative answer to this question
was obtained by V.\ Piterbarg in his thesis \cite{piterbarg1997expansions}, in the setting of isotropic stochastic flows.
One should also mention a paper (\cite{PPIntermittencyTracerGradient1999a}) by Piterbarg, Piterbarg where the authors study the correlation
functions for the random fields $\{x_i, i\in \Z\}$ and $\{a_i,i\in \Z\}$.
Also, in \cite{NTWeakConvergenceLocalized2015} the authors considered compositions of \iid monotone maps on a circle that are small perturbation of the identity map and showed that the resulting flow
after proper rescaling converges weakly to the coalescing Brownian flow. 

The next important question
is the rate of convergence to the fixed point. To the best of our knowledge this question was not
studied before. In this paper, we prove that under mild conditions on spatial decorrelation for the
distribution of the maps~${\Psi_{1,0}(x) -x}$, the iterates of maps 
$\Psi_{n,0}=\Psi_{n,n-1}\circ \Psi_{n-1,n-2} \circ, \dots \circ \Psi_{1,0}$
converges to the above fixed point distribution with the power law rate $n^{-K}$ for some $K>0$.

The power-law rate can be interpreted as the stability of a fixed point. To explain this, one
can define the renormalization transformation $R$ by iterating just two independent copies. Then
$R^m=R\circ R \dots \circ R$ corresponds to $n=2^m$ iterates. The power-law rate in $n$ implies that 
$R^m\Psi$ is $2^{-Km}$ close to the fixed point. This means that the fixed point for $R$ provided by $\Phi_{1,0}$
is exponentially stable.

We note that renormalization transformation corresponding to iteration of maps (non-random) is a
powerful tool in the theory of dynamical systems. At the same time, assuming that the maps
are random and statistically independent brings us closer to the realm of Gaussian fixed point
in the CLT setting.

We believe that the fixed point corresponding to the coalescing Brownian motions can be generalized
for random maps with non-diffusive spatial scaling. Such constructions will require the random maps to be
correlated in time. It is natural to expect that for any $0<\alpha<1$ one can construct
a fixed point with the spatial scaling $n^{-\alpha}$. Moreover, this fixed point may be also stable
provided that the iterates have the same spatial scaling behavior. It is tempting to try to construct such
fixed points using coalescing fractional Brownian motions. However, at present such an object has
not been constructed rigorously.

Our interest in $\alpha$ different from $1/2$ is motivated by
various physical applications. In particular, the fixed point for $\alpha=2/3$ should be closely related
to the KPZ phenomenon.

This paper is organized as follows.
In section~\ref{sec:setting} we introduce our assumptions on the monotone maps and present the main results.
In section~\ref{sec:coal-brown-motion} we review the construction and some basic properties of the coalescing Brownian motions that will be used in our proofs.
In section~\ref{sec:proof-main-theorem} we give the proof of the main theorem.
In section~\ref{sec:appendix} we collect some results on the estimates of hitting times and extreme values of Brownian motions and Markov chains.

\bigskip

\textbf{Notations:} we will write a $\R^m$-vector as~$\mathbf{x} = (x_1,x_2, \ldots, x_m)$, and a stochastic process $T \to \R^m$ as 
$\mathbf{X}_t = (X^{(1)}_t, X^{(2)}_t, \ldots, X^{(m)}_t)$, $t \in T$.  For a process $(X_t)_{t \in T}$, we may also write $X(t) = X_t$.
For~$a$ and~$b \in \R$, we denote the minimum and the maximum of~$a$ and~$b$ by~$a \wedge b$ and $a \vee b$, respectively. We will use $\Cc_R(I)$ or~$\Dc_R(I)$ to denote the functional space of continuous or c\`adl\`ag functions from~$I$ to~$R$; when~$R = \R$ we will simply write~$\Cc(I)$ or~$\Dc(I)$.
We will use~$C$, $C_i$ for generic absolute constants in estimates, and their values may vary from line to line.

\section{Assumptions and the main result}
\label{sec:setting}
We recall that $(\Psi_{i+1,i})_{i \in \Z}$ is a sequence of \iid monotone piecewise constant maps.
Naturally, we assume that the centered processes $\psi_n(x) = \Psi_{n,n-1}(x) -  x$ is ergodic in~$x$ for the spatial homogeneity.
In  addition, we will assume  the following conditions hold for the distribution of $\Psi_{n,n-1}$:
\begin{description}
\item[(A1)\label{it:minimum-gap}]
For any~$A > 0$, there 
exists~$\rho>0$ and~$l>0$ such that
\begin{equation*}
\Pp \Big(  \Psi_{l,0}(0) = \Psi_{l,0}(A)\Big) \ge \rho.
\end{equation*}
\item[(A2)\label{it:first-second-moment}] Let $\mu$ the distribution of~$\psi_1(0)$.  Then 
\begin{equation*}
 \int_{\R} x \mu(dx) = 0, \quad \int_{\R} x^2 \mu(dx) = 1.
\end{equation*}
\item[(A3)\label{it:higher-moments}] 
  For some~$\alpha > 3$,
\begin{equation*}
\int_{\R} |s|^\alpha \mu(ds) < \infty
\end{equation*}
\item[(A4)\label{it:decoupling}] For some constants~$\beta > 1$ and~$b > 0$,
  the process~$\psi_1(\cdot)$ is~$\phi$-mixing with~$\phi(s) \le 1 \wedge (bs^{-\beta})$, that is, for any~$s > 0$, 
\begin{equation*}
  \big| \Pp ( B|A ) - \Pp (B) \big| \le \phi(s),
  \quad \forall
  A \in \Fc_{\le 0}, \ \Pp(A) \ge 0 \text{ and } B \in \Fc_{\ge s},
\end{equation*}
where~$\Fc_{\le 0} = \sigma(\psi_1(x), \ x \le 0)$ and~$\Fc_{\ge s}= \sigma \big( \psi_1(x), \ x \ge s  \big)$.
\end{description}
The assumption~\ref{it:minimum-gap} guarantees that coalescence occurs in finite time from any two starting positions.
The assumption~\ref{it:first-second-moment} is a normalizing condition.  The assumptions~\ref{it:higher-moments} and \ref{it:decoupling} can be satisfied when~$\psi_1(\cdot)$ has finite range of
dependence (see section~\ref{sec:finite-range-depend}), and in general they mean that the process $\psi_1(\cdot)$ has sufficiently fast decorrelation.

Next we will describe the topology of the convergence.
Let~$\mathcal{M}$ be the space of increasing functions on~$\R$.
We will introduce a generalized L\'evy metric on~$\mathcal{M}$: for~$f, g \in \mathcal{M}$, define 
\begin{equation*}
  d(f,g) = \sum_{B=1}^{+\infty} \frac{\tilde{d}( \chi_B\circ f ,  \chi_B\circ g) \wedge 1}{2^B} ,
\end{equation*}
where~$\chi_B(x) = (-B) \wedge x \vee B$ are cutoff functions, and~$\tilde{d}$ is the L\'evy metric
\begin{equation*}
\tilde{d}(f, g ) = \inf \{  \eps > 0: \ f(x-\eps) - \eps \le g(x) \le f(x+\eps) + \eps, \ \forall x \in \R \}.
\end{equation*}
The metric~$d$ will not distinguish between the right-continuous and left-continuous versions of an element in~$\mathcal{M}$.

Let $\tilde{\Psi}_n(\cdot) = n^{-1/2} \Psi_{n,0} (n^{1/2}\cdot)$ and we recall that~$\Psi_{n,0}$ is the transport map given by the coalescing Brownian motions.
\begin{theorem}
  \label{thm:main-algebraic-conv-to-cBM}
  There exists a common probability space such that for some constants~$K_1, K_2>0$,
\begin{equation*}
  \E  d(\tilde{\Psi}_n, \Phi_{1,0}) \le K_1 n^{-K_2}.
\end{equation*}
That is, the Wasserstein-$1$ distance between the distributions of~$\tilde{\Psi}_n$ and~$\Phi_{1,0}$ with respect to the metric~$d$ decays as~$O(n^{-K_2})$.
\end{theorem}

\section{Coalescing Brownian motion}
\label{sec:coal-brown-motion}
In this section we will discuss the construction of the coalescing Brownian motion and the coalescing Brownian flow, mostly following the origin treatment in \cite{ArratiaCBM}.

We first define the coalescing Brownian motion of finitely many particles.
Let~${A \subset \R}$ with~$|A| = m$ be the set of starting positions.
The coalescing Brownian motions on~$A$ is a continuous process~$\mathbf{c}(t): [0,\infty) \to \R^m$, with the following properties: 
\begin{enumerate}[1)]
\item for~$i = 1, \ldots, m$, the process~$c^{(i)}(t)$ is a Brownian motion; 
\item the Brownian motions start from $A$: $A = \{  c^{(i)}(0): i = 1,\ldots,m \}$;
\item (coalescing property) any two processes $c^{(i)}$ and~$c^{(j)}$ are independent before their coalescence time 
\begin{equation*}
\tau_{ij} = \inf \{  t\ge 0: c^{(i)}(t) = c^{(j)}(t) \},
\end{equation*}
and $c^{(i)}(t) = c^{(j)}(t)$ for~$t \ge \tau_{ij}$.
\end{enumerate}
It is clear from the coalescing property and one dimensionality that the order of particles is preserved, namely, if~$c^{(i) }(s) \le c^{(j)}(s)$ and~$s < t$, then~$c^{(i)}(t) \le c^{(j)}(t)$.
For this reason, we can replace the state space $\R^m$ by the set 
\begin{equation*}
  \Delta_m = \{ (x_1, \ldots, x_m): x_1 \le \ldots \le x_m \}.
\end{equation*}

There are more than one ways to construct the coalescence Brownian motion of finitely many particles as a unique stochastic process on $\mathcal{C}_{\Delta_m}[0,\infty)$.  For the purpose of constructing coupling that will be used later in the proof, we will construct
the finite-particle coalescence Brownian motion from independent Brownian motions with a ``collision rule'': independent Brownian particles are give ranks initially, and when two or more particles collides,
all particles will stick together and follow the motion of the particle with the highest rank.

More precisely, let~$m\ge 2$ and we denote by~$\Sigma_m$ the set of permutation of~$1,\ldots, m$.
We fix a permutation $\sigma \in \Sigma_m$; this will assign a ranking among~$m$ particles by putting the~$i$-th particle at the~$\sigma(i)$-th place.
Let $\mathbf{B}$ be an $m$-dimensional Brownian motion starting from~$\mathbf{x}\in\R^m$.
We introduce the ``follower functions'' $f_i \in \Dc_{\{ 1,\ldots,m \}}[0,\infty)$ to record the highest-ranked particle in each cluster.
Namely, $f_i(t) = j$ means that at time~$t$, particle~$i$ is following particle~$j$; we may call $j$ the leader of~$i$ at time~$t$.
We consider the the functional
\begin{align*}
  \phi_{\sigma}: \Cc_{\R^m}[0,\infty) &\to \Cc_{\R^m}[0,\infty) \\
   \BB &\mapsto \mathbf{c} = \phi_{\sigma}(\BB)
\end{align*}
defined by 
\begin{equation*}
c^{(i)}(t) = B^{(f_i(t))}(t), \quad i = 1, \ldots, m.
\end{equation*}
The functions~$f_i$ have the following properties.
\begin{enumerate}[1)]
\item \label{item:init-cond} For $i = 1, \ldots, m$, we have~$f_i(0) \le i$, and the strictly inequality holds only when particle $i$ starts at the same position as another particle with a higher rank.
\item \label{item:increasing-rank}
  The map $t \mapsto \sigma^{-1} \big( f_i(t) \big)$ is non-increasing, meaning that particles will not change their leaders until collision with another particle with a higher rank.
\item If~$j = f_i(t)$, then~$j = f_j(t)$. 
\item \label{item:coal-property} If $f_i(s) = f_j(s)$, then $f_i(t) = f_j(t)$ for~$t \ge s$.
\end{enumerate}
The following result gives a construction of the finite-particle coalescing Brownian motion using the collision rule.
\begin{theorem}[\cite{ArratiaCBM}]
  \label{thm:functional-of-BM}
  Let~$m \ge 2$,  $\sigma \in \Sigma_m$ and~$\BB \in \mathcal{C}_{\R^m}[0,\infty)$.
  \begin{enumerate}[1)] 
  \item The follower functions~$f_i$ and the functional $\phi_{\sigma}$ satisfying~\ref{item:init-cond}-\ref{item:coal-property} are uniquely determined by~$\BB$.
\item If~$\BB(0)\in \Delta_m$, then~$\phi_{\sigma}(\BB) \in \Delta_m$, that is, the functional~$\phi_{\sigma}$ preserves orders.
\item If~$\BB$ is an $m$-dimensional Brownian motion, then~$\phi_{\sigma}(\BB)$  has the distribution of coalescing Brownian motions of $m$ particles. 
\end{enumerate}

\end{theorem}

Next we will consider the coalescing Brownian motion from~$\R$.
A priori it is not clear how to define such a system, since it is impossible to keep track of the collision times.
The following result was first obtained in \cite{ArratiaCBM}.
\begin{theorem}[\cite{ArratiaCBM}]
\label{thm:existence-of-cBM}
There exists a random process~$\mathbf{c} = \Big( c(x,t), \ x \in \R, t \ge 0 \Big)$ such that for any finite~$A$ with~$|A|=m$, the process 
\begin{equation*}
  \mathbf{c}^A = \Big(  c(x,t), \ x \in A, t \ge 0 \Big)
\end{equation*}
is coalescing Brownian motions of $m$ particles and $c(x,0) =x $ for $x\in A$.

Moreover, for every~$t>0$, the set $\{ c(x,t): x \in \R \}$ is a discrete set.
\end{theorem}
The transport map~$\Phi_{n,0}$ introduced at the beginning of this paper is  given by 
\begin{equation*}
\Phi_{n,0}(x) = c(x,n).
\end{equation*}




\section{Proof of the main Theorem}
\label{sec:proof-main-theorem}
Our proof of Thoerem~\ref{thm:main-algebraic-conv-to-cBM} relies on the following lemma that establishes a coupling between  the coalescing Brownian motions and the ``coalescing random walks''.
\begin{lemma}
  \label{lem:key-coupling-lemma}
    Let~$m \ge 1$ and~$\gamma,\kappa,\delta$, $\theta$ satisfy 
\begin{equation}\label{eq:range-for-indices}
  0<\gamma<\kappa<\theta < \frac{1}{2}, \quad
  0 < \delta < \min(\gamma\alpha-1, \gamma \beta-1, \kappa-\gamma, \theta - \kappa).
\end{equation} 
Then, for sufficiently large~$n$ and for every~$\mathbf{x}\in \R^m$ such that~
\begin{equation}
\label{eq:cond-on-initial-x}
x_i + n^{\theta}< x_{i+1},\quad i = 1, \ldots, m-1,
\end{equation}
there exists a Markov chain~$( \Ss_k)_{k \in \N}$ on~$\R^m$ with $\Ss_0 = \mathbf{x}$ and transition probability 
\begin{equation}\label{eq:distr-of-coal-RW-m-traj}
  \Pp \big(  \Ss_{k+1} -\Ss_k \in d\mathbf{z}' \mid \Ss_k = \mathbf{z}
  \big) = \Pp \big(  \psi_1(z_1) \in dz_1',\ldots, \psi_1(z_m) \in dz_m' \big),
\end{equation}
and an $m$-particle coalescing Brownian motion~$(\WW_t)_{t\ge 0}$ with~$\WW_0 = \mathbf{x}$, such that
\begin{equation}
\label{eq:n-to-one-quater-dist-m-traj}
\Pp \Big(  \max_{0 \le k \le n} \|\Ss_k - \WW_k\|_{\infty} > 3n^{\kappa} \Big)
\le Kn^{-\delta}m\ln m
\end{equation}
holds for some constant~$K$.
The constant~$K$ depends on the distribution of~$\psi_1$ and the choice of~$\delta$.
\end{lemma}

The transition probability~\eqref{eq:distr-of-coal-RW-m-traj} and the initial condition~$\Ss_0 = \mathbf{x}$ imply
\begin{equation*}
  \big( S^{(1)}_k, \ldots, S^{(m)}_k \big)_{k \in \N} \disteq
  \big( \Psi_{k,0}(x_1), \ldots, \Psi_{k,0}(x_m)
  \big)_{k \in \N}.
\end{equation*}
he process~$(\Ss_k)$ will be called ``coalescing random walks'' since for fixed~$x$, $\big(\Psi_{n,0}(x) \big)_{n \ge 0}$ is a random walk with \iid steps distributed as~$\mu$:
\begin{equation*}
\Psi_{n+1,0}(x) - \Psi_{n,0}(x) = \psi_n( \Psi_{n,0}(x)) \disteq \psi_1(0).
\end{equation*}
We note that unlike the coalescing Brownian motions, the random walks starting from different locations are not independent before coalescence.

An immediate corollary of Lemma~\ref{lem:key-coupling-lemma} is the existence of the following coupling between~$\tilde{\Psi}_n$ and~$\Phi_{1,0}$.
\begin{lemma}
  \label{lem:main-lemma}
  Let~$\gamma,\kappa,\delta$, $\theta$ satisfy~\eqref{eq:range-for-indices}.
Then for sufficiently large~$n$ and $\mathbf{y} \in \Delta_m$ with~$y_i+n^{\theta-\frac{1}{2}} < y_{i+1}$,
we have a coupling of~$\Psi_n$ and~$\Psi_{\infty}$ such that
\begin{equation*}
\Pp \Big(  \max_{1 \le i \le m} |\tilde{\Psi}_n (y_i) - \Phi_{1,0}(y_i)| > 3n^{\kappa-\frac{1}{2}} \Big) \le C n^{-\delta}m\ln m.
\end{equation*}
\end{lemma}

Using Lemma~\ref{lem:main-lemma} we can prove the main theorem.

\begin{proof}[Derivation of Theorem~\ref{thm:main-algebraic-conv-to-cBM} from Lemma~\ref{lem:main-lemma}]
  Let~$D>0$ be a large number.
  Let~$y_1<y_2<\ldots <y_m$ be evenly spaced on the interval~$[-D\ln n, D\ln n]$: 
\begin{equation*}
y_1 = -D\ln n, \quad y_m = D\ln n, \quad y_{i+1}-y_i \equiv \Delta y = \frac{2D \ln n}{m-1}.
\end{equation*}
Assume that~$\Psi_n$ and~$\Psi_{\infty}$ are given by the coupling in Lemma~\ref{lem:main-lemma} and let~$\Lambda$ be the event~
\begin{equation*}
\Lambda = \max_{1 \le i \le m} |\tilde{\Psi}_n (y_i) - \Phi_{1,0}(y_i)| \le 3n^{\kappa-\frac{1}{2}}.
\end{equation*}
For~$B = [D\ln n]$, by the definition of~L\'evy metric,  we have 
\begin{align*}
  \E d(\tilde{\Psi}_n, \Phi_{1,0})  &\le
                                     \E \ONE_{\Lambda}\cdot\tilde{d} \big( \chi_B \circ \tilde{\Psi}_n, \chi_B\circ \Phi_{1,0} \big) +P(\Lambda^c) + 2^{-B}\\
                                   &\le \big( \Delta y  \big) \vee (3n^{\kappa-\frac{1}{2}}) + Cn^{-\delta} m \ln m + 2n^{-D\ln 2}.
\end{align*}
Let us first choose~$D$ large so that~$D\ln 2 > 2$, and then choose~$m=[n^{\frac{1}{2} - \theta}]$ so that~$\Delta y \in [C_1n^{\theta-\frac{1}{2}} \ln n, C_2 n^{\theta-\frac{1}{2}} \ln n]$.
The rightmost term is then bounded by 
\begin{equation*}
  C \cdot \big( n^{\theta-\frac{1}{2}}\ln n + n^{\frac{1}{2}-\delta-\theta}\ln n \big).
\end{equation*}
To optimize this, we should have $\theta+\delta-1/2 = 1/2-\theta$, i.e.,
\begin{equation}
\label{eq:2}
\theta = \frac{1}{2} - \frac{\delta}{2}.
\end{equation}
Then the statement of the theorem is true for~$K_2 = \frac{\delta}{2}>0$.

To finish the proof, we still need to find $\gamma,\kappa,\theta,\delta$ satisfying \eqref{eq:range-for-indices} and~\eqref{eq:2}.
We introduce a small parameter~$\eps>0$ and set~$\gamma,\kappa,\theta$ in terms of $\eps$ and $\delta$: 
\begin{equation*}
  \gamma = \frac{1+\delta+ \eps}{ \alpha \vee \beta}, \quad
  \kappa = \gamma + \delta + \eps, \quad
  \theta = \kappa + \delta + \eps,  
\end{equation*}
Then all the relations in \eqref{eq:range-for-indices} are satisfied except $\theta < \frac{1}{2}$; but $\theta<\frac{1}{2}$ is implied by~\eqref{eq:2}, so it remains to ensure~\eqref{eq:2}.  Expressing~\eqref{eq:2} in terms of
$\eps$ and $\delta$, we have 
\begin{equation*}
\theta  = \gamma + 2(\delta+\eps) = \frac{1+\delta+\eps}{\alpha \vee \beta} + 2(\delta+\eps) = \frac{1}{2} - \frac{\delta}{2}.
\end{equation*}
Then 
\begin{equation*}
\eps=   \frac{\frac{1}{2} - \frac{1}{\alpha \vee \beta}  - \Big(  \frac{5}{2} + \frac{1}{\alpha \vee \beta} \Big) \delta}{ 2 + \frac{1}{\alpha \vee \beta}}.
\end{equation*}
Since~$\alpha \vee \beta \ge \alpha> 3$ by~\ref{it:higher-moments}, we can always choose $\delta>0$ to be sufficiently small to ensure~$\eps>0$, and this finishes the proof.
\end{proof}

The rest of this section will be devoted to the proof of Lemma~\ref{lem:key-coupling-lemma}.

\subsection{Almost surely invariance principle}
Lemma~\ref{lem:key-coupling-lemma} estimates the distance between coalescing random walks and coalescing Brownian motions.
In the $m=1$ case, such results are known as the ``almost surely invariance principle'', which provides quantitative estimates of how well Brownian motions approximate the partial sums of \iid r.v.'s.
Lemma~\ref{lem:key-coupling-lemma} is a consequence of the almost surely invariance principle and a careful analysis of the collision times.

We will cite the following result for the almost surely invariance principle, proven  by Koml\'{o}s, Major and Tusn\'{a}dy (\cite{KMT2}).
\begin{theorem}
  \label{thm:KMT}
  Let~$\nu$ be a probability distribution on~$\R$ such that 
\begin{equation*}
  \int_{\R} x \, \nu(dx) = 0, \quad
  \int_{\R} x^2 \, \nu(dx) = 1, \quad
  \int_{\R} |x|^r < \infty,
\end{equation*}
where~$r>3$.
Let~$\gamma \in (r^{-1},\frac{1}{2})$.
For sufficiently large~$n$, 
there exists a random walk~$(S_k)_{0 \le k \le n}$ with step distribution~$\nu$ and a Gaussian random walk~$(W_k)_{0\le k \le n}$ with step distribution~$\Nc(0,1)$ on a common probability space, such that
\begin{equation}\label{eq:n-gamma-close}
  \Pp \Big(  \max_{0 \le k \le n} |S_k -
  W_k| > n^{\gamma} \Big) \le \frac{d_1}{n^{r\gamma-1}}
\end{equation}
for some constant~$d_1$ depending only on~$\nu$.

When~$\nu$ has exponential moments, namely, $\E e^{\eta|X_1|} < \infty$ for some~$\eta>0$, the following constraint 
\begin{equation}\label{eq:log-close}
  \Pp  \Big( \max_{0 \le k \le n} |S_k - W_k| > d_3\log n + x
  \Big) \le d_4e^{-d_2 x}
\end{equation}
can be satisfied for some constants~$d_2$, $d_3$ and~$d_4$ depend only on~$\nu$.

Moreover, in both cases, the processes $S$ and~$W$ can be related via
\begin{equation*}
(T_1,\ldots,T_n) = \Xi \big( (S_1,\ldots,S_n), U \big),
\end{equation*}
where $U$ is a r.v.\ with distribution~$\mathrm{U}[0,1]$ and independent of~$S$, and
the map
\begin{equation*}
\Xi:\R^{n} \times [0,1] \to \R^n
\end{equation*}
is deterministic and depending only on~$\nu$.
\end{theorem}

As an immediate corollary of Theorem~\ref{thm:KMT}, the Gaussian random walk~$W$ can be replaced by a Brownian motion on~$[0,\infty)$.
\begin{lemma}
\label{lem:cor-of-KMT}
Assume the conditions in Theorem~\ref{thm:KMT}.
Then the same estimates \eqref{eq:n-gamma-close} or \eqref{eq:log-close} hold for~$B = \big( B_t \big)_{t \ge 0}$ being a standard Brownian motion, and 
\begin{equation*}
(B_t)_{t \ge 0} = \Xi \big(  (S_1,\ldots, S_n), U \big)
\end{equation*}
for some deterministic functional~$\Xi: \R^n \times [0,1] \to \Cc[0,\infty)$.
\end{lemma}

\subsection{Finite range of dependence}
\label{sec:finite-range-depend}
In this section we will assume that
$\psi_1(\cdot)$ has finite dependence range of~$L$, namely, there is a constant~$L>0$ such that
\begin{equation}
  \label{eq:finite-dependence-range}
  \forall x_0, \ 
   \big(\psi_1(x), \ x \le x_0 \big) \text{ and } \big( \psi_1(x), \ x\ge x_0 + L \big)
  \text{ are independent.}
\end{equation}
Clearly, condition~\eqref{eq:finite-dependence-range} implies that~\ref{it:decoupling} holds for arbitrarily large~$\beta$.
Also as a consequence of~\eqref{eq:finite-dependence-range}, $\mu$ must be supported on~$[-L,L]$ and hence $\alpha$ in~\ref{it:higher-moments} can be arbitrarily large.
Hence the condition~\eqref{eq:range-for-indices} can be simplified to 
\begin{equation}\label{eq:simple-cond-for-indices}
0 < \kappa < \theta < \frac{1}{2}, \quad 0 < \delta < \kappa \wedge (\theta-\kappa).
\end{equation}

In this section the following version of Lemma~\ref{lem:key-coupling-lemma} will be proved. 
\begin{lemma}
\label{lem:key-coupling-lemma-finite-range}
Let~$m\ge 2$.  Assume~\eqref{eq:finite-dependence-range} and let~$\delta,\kappa,\theta$ satisfy \eqref{eq:simple-cond-for-indices}.  Then there is a coupling of the coalescing random
walks~$\Ss$ and coalescing Brownian motions~$\WW$ satisfying the conditions in Lemma~\ref{lem:key-coupling-lemma}.
\end{lemma}

Under the condition~\eqref{eq:finite-dependence-range},  the random walks~$S^{(i)}$ as components of~$\Ss$ are independent until they get closer than distance~$L$.
In section~\ref{sec:coal-brown-motion} we see that coalescing Brownian motions can be constructed from independent Brownian motions using the functional~$\phi_{\sigma}$; similarly, the
coalescing random walks can be constructed out of independent random walks via some map~$\tilde{\phi}_{\sigma}$.

The proof strategy is to transfer the coupling between random walks and Brownian motions given by Lemma~\ref{lem:cor-of-KMT} to 
the coupling between their coalescing versions via~$\phi_{\sigma}$ and~$\tilde{\phi}_{\sigma}$, 
summarized in the following diagram:
\begin{equation}\label{eq:coupling-scheme}
  \begin{tikzcd}[row sep=1cm, column sep=3cm]
    \text{$\Ss$: coalescing\ RWs }  &  \text{$\WW$: coalescing\ BMs } \\
    \text{$\bar{\Ss}$: independent\ RWs} \arrow[u, shift left, "\displaystyle{\tilde{\phi}_{\sigma}}"]
    & \text{$\bar{\WW}$: independent\ BMs}    \ar[leftrightarrow, l , shift left,"\text{\large Lemma~\ref{lem:cor-of-KMT}}"'] \ar[u, shift left, "\displaystyle{\phi_{\sigma}}"'].
  \end{tikzcd}
\end{equation}

Intuitively, Lemma~\ref{lem:key-coupling-lemma} would follow from the argument 
\begin{equation}\label{eq:13}
  \bar{\Ss} \approx \bar{\WW}, \ \tilde{\phi}_{\sigma} \approx \phi_{\sigma}
  \quad \Rightarrow \quad \Ss = \tilde{\phi}_{\sigma}(\mathbf{\Ss}) \approx \WW =\phi_{\sigma} (\bar{\WW}).
\end{equation}
However, maps like~$\phi_{\sigma}$ cannot be continuous on the space~$\Cc_{\R^m}[0,\infty)$, as the coalescence time of trajectories is very unstable under small perturbation.
The main technicality in this section is to show that the functional~$\phi_{\sigma}$ is H\"older-continuous except on a subset of~$\Cc_{\R^m}[0,\infty)$ that has small probability under the Wiener measure.
Precisely, we will establish results of the following kind: for~an $m$-dimensional Brownian motion~$\BB$, with probability at least~$1-Cmn^{-\delta}$, 
\begin{equation*}
  \| \mathbf{F}  -  \BB\|_{\Cc_{\R^m}[0,n]} \le n^{\gamma}
  \Rightarrow \| \phi_{\sigma}(\mathbf{F}) - \phi_{\sigma}(\BB)\|_{\Cc_{\R^m}[0,n]} \le n^{\kappa}.
\end{equation*}

\medskip 
First, let us give the construction of the processes in \eqref{eq:coupling-scheme}.

Suppose that on a common probability space there are the following three independent families of random objects:
\begin{enumerate}[1)]
\item a Markov chain $\big( \Ss_k \big)_{k \in \N}$ with initial condition~$\mathbf{x}$ and transition probability~\eqref{eq:distr-of-coal-RW-m-traj}, 
\item independent r.v.'s~$\xi_k^{(i)}$, $1 \le i \le m$, $1 \le k $, with distribution~$\mu$,
\item independent r.v.'s~$U_i$, $1 \le i \le m$, with distribution~$\mathrm{U}[0,1]$.
\end{enumerate}

We fix $\sigma \in \Sigma_m$.
For $1 \le i \le m$, we define
\begin{equation}\label{eq:def-of-tau-s-i}
    \tau_L^{i,s} = \min \{  k \in \N: \exists j \text{ s.t. } \sigma(j) < \sigma(i), \  |S^{(i)}_k - S^{(j)}_k| \le L \}, 
  \end{equation}
  where~$L$ is introduced in~\eqref{eq:finite-dependence-range}.
Then we define the processes~$\bar{S}^{(i)}$ by
\begin{equation}\label{eq:def-S-bar-i}
    \bar{S}^{(i)} _k = S^{(i)}(k \wedge \tau_L^{i,s}) + \ONE_{k > \tau_L^{i,s}} \sum_{j = 1}^{k - \tau_L^{i,s}} \xi_j^{(i)}, \quad 0 \le k , \ 1 \le i \le m
\end{equation}  
Clearly, $\bar{S}^{(i)}$ are independent random walks with step distribution~$\mu$.
Moreover, from the definition we have 
\begin{equation}
\label{eq:S-equal-S-bar}
S^{(i)}_k = \bar{S}^{(i)}_k , \quad 0 \le k \le \tau_L^{i,s} \wedge n.
\end{equation}

By Lemma~\ref{lem:cor-of-KMT},
\begin{equation*}
\bar{W}^{(i)} = \Xi \big( (\bar{S}^{(i)}_1, \ldots, \bar{S}^{(i)}_n ), U_i \big), \quad i = 1, \ldots, m, 
\end{equation*}
are independent standard Brownian motions starting from~$x_i$.
Finally, we construct~$\WW$ by $\WW = \phi_{\sigma}(\bar{\WW})$.
Then~$\WW = (W^1,\ldots, W^m)$ are coalescing Brownian motions of~$m$ particles with~$\WW_0 = \xx$. 
By~\eqref{eq:log-close}, we have the following estimate.
\begin{lemma}
  \label{lem:dist-bar-S-and-W}
  For some constant~$K_1 \ge 3$, 
\begin{equation}
\label{eq:dist-bar-S-and-W-event}
\max_{0 \le k \le n} \| \bar{\WW}_k - \bar{\Ss}_k\|_{\infty} \le K_1 \log n
\end{equation}
holds with probability at least~$1 - Cme^{-10n}$.
\end{lemma}
We will also use the following simple bound on displacement of Brownian motions in unit intervals.
\begin{lemma}
  \label{lem:BM-in-unit-interval}
  For the $m$-dimensional Brownian motion~$\mathbf{\WW}$, 
\begin{equation}\label{eq:BM-in-unit-interval}
 \max_{0 \le k \le n-1} \sup_{t \in [k,k+1]} \|\bar{\WW}_t - \bar{\WW}_k\|_{\infty} \le  \log n, 
\end{equation}
holds with probability at least~$1-Cmn^{-2}$.
\end{lemma}

\medskip

Next we will prove that~\eqref{eq:n-to-one-quater-dist-m-traj} holds for the above constructed
$\Ss$ and~$\WW$.
Let us write~$M = K_1 \log n$ and assume that~$n$ is large enough so that 
\begin{equation*}
L \le M = K_1 \log n \le n^{\kappa}.
\end{equation*}

The follower functions~$f_i:[0,\infty) \to \{ 1,\ldots,m \}$ introduced in~\ref{sec:coal-brown-motion} determine the mapping~$\WW=\phi_{\sigma}(\bar{\WW})$ by
\begin{equation*}
W^{(i)}_t = \bar{W}^{f_i(t)}_t, \quad i = 1, \ldots, m, \  0 \le t.
\end{equation*}
We can similarly define follower functions~$g_i: \N \to \{ 1,\ldots,m \}$ for~$\Ss$: $g_i(k)=j$ is the index with the highest rank such that
\begin{equation*}
S^{(i)}_k \neq S^{(j')}_k, \quad \sigma(j') < \sigma(j).
\end{equation*}
The distance~$|S^{(i)}_k - W^{(i)}_k|$ can be controlled by~\eqref{eq:dist-bar-S-and-W-event} if we have
\begin{equation}
\label{eq:follower-functions-agree}
g_i(k) = f_i(k) = j, \quad g_i(k) \le \tau^{j,s}_L,
\end{equation}
since in such case, 
\begin{equation*}
W^{(i)}_k = \bar{W}_k^{(j)}, \quad S^{(i)}_k = \bar{S}_k^{(j)}.
\end{equation*}
Let~$\eps>0$ be a small number that satisfies 
\begin{equation}
\label{eq:cond-for-eps}
\eps < \frac{1}{10} \big( \kappa \wedge (\theta-\kappa) - \delta \big).
\end{equation}
Our strategy is to show that for every~$i$, those~$k$ that violates~\eqref{eq:follower-functions-agree} can be contained in intervals of length at most~$n^{2\kappa-\eps}$.
Let us illustrate how this strategy works in the case~$m=2$ and~$3$.

For two intervals~$I$ and~$J$, we write~$I \le J$ if~$x \le y$ for all~$x\in I$ and~$y \in J$; expressions like~$I < J$, $I \le x$ will have similar meaning.
\begin{lemma}  \label{lem:m-2-existence-of-Ii}
  Let~$\sigma = (2 \ 1)$ and~$I = \big[\tau, \tau + [n^{2\kappa-\eps}] \big] \cap [0,n]$, where
\begin{equation*}
\tau = \min \{  k \in \N: |\bar{W}^{(1)}_k- \bar{W}^{(2)}_k| \le 3M\}
\end{equation*}
and $[\cdot]$ denotes the integer part.
Then with probability at least~$1-Cn^{-\delta}$, the following holds:
\begin{equation}\label{eq:dist-S-W-bar-small-2}
  \max_{0 \le k \le n} \| \bar{S}^{(i)}_k  - \bar{W}^{(i)}_k\| \le M,\quad i = 1, 2; 
\end{equation}
  \begin{subequations}\label{eq:properties-of-I-m-2}
  \begin{align}
\label{eq:before-coal-time-m-2}      W_t^{(1)} = \bar{W}_t^{(1)}, \quad S^{(1)}_k = \bar{S}^{(1)}_k,  &\qquad 0\le t,k \le I, \\
\label{eq:after-coal-time-m-2}  W_t^{(1)} = W_t^{(2)}, \quad S^{(1)}_k = S^{(2)}_k, &\qquad  I\le t,k \le n; 
  \end{align}
\end{subequations}
  \begin{equation}
\label{eq:deviation-in-I-m-2}
\max_{k \in I} |S^{(1)}_k - S^{(1)}_{\tau}|  \le n^{\kappa}, \quad
 \sup_{t \in I} |W^{(1)}_t - W^{(1)}_{\tau}|  \le n^{\kappa}.
\end{equation}
Furthermore, \eqref{eq:dist-S-W-bar-small-2}, \eqref{eq:properties-of-I-m-2} and~\eqref{eq:deviation-in-I-m-2} imply 
\begin{equation}\label{eq:19}
\|S^{(i)}_k - W_k^{(i)}\| \le 
\begin{cases}
M, & i=1 \text{ or } i=2, \ k \not\in I,  \\
M + 2n^{\kappa}, & i = 2, \ k \in I.     \\
\end{cases}
\end{equation}
and hence~\eqref{eq:n-to-one-quater-dist-m-traj} is satisfied.
\end{lemma}
Using the follower functions~$f_i$ and~$g_i$, we can also write~\eqref{eq:properties-of-I-m-2} as\begin{align*}
 f_1(t)  = g_1(k) = 
\begin{cases}
1, & 0 \le t, k \le I,  \\
2,   & I \le t,k \le n,  \\
\end{cases} \quad f_2(t) =  g_2(k) \equiv 2, \ t,k \in [0,n].
\end{align*}
\begin{proof}
  By \eqref{eq:dist-bar-S-and-W-event} and Lemma~\ref{lem:BM-in-unit-interval}, there is an event ~$\Lambda$ such that~\eqref{eq:dist-S-W-bar-small-2} and~\eqref{eq:BM-in-unit-interval} hold on $\Lambda$ and~$\Pp (\Lambda) \ge 1 - Cn^{-2}$.
  
  Assume that we are on the event~$\Lambda$.
  Let us consider the following random times: for~$h \ge 0$, 
  \begin{align}
      \label{eq:def-tau-S}
\tau_h^s &= \min \{ k \in \N: |S^{(2)}_k - S^{(1)}_k| \le h \}\\
\label{eq:def-tau-W-bar}
    \tilde{\tau}^w_h &=  \inf \{ t \ge 0: |\bar{W}^{(2)}_t -  \bar{W}^{(1)}_t | \le  h  \}, \\
    \label{eq:tau-h-w}
      \tau_h^w &= \min \{  k \in \N: |\bar{W}^{(2)}_k -  \bar{W}^{(1)}_k | \le h \}.
\end{align}
In particular, we have~$\tau_{3M}^w = \tau$ and~$\tau_L^s = \tau^{1,s}_L$, where $\tau^{1,s}_L$ is introduced in \eqref{eq:def-of-tau-s-i}.

Clearly, the random times $\tau_h^s$, $\tilde{\tau}_h^w$, $\tau_h^w$ are all non-decreasing as~$h$ decreases.
Moreover,  they satisfy the following relations on the event~$\Lambda$:
\begin{subequations}\label{eq:order-of-stopping-times}
  \begin{align}
\label{eq:1} \tilde{\tau}_h^w &  \le \tau_h^w, \quad 0 \le h. \\ 
  \label{eq:6} \tau_{h+2M}^w \wedge n &\le \tilde{\tau}_h^w \wedge n, \quad  0 \le h, \\
  \label{eq:7} \tau_{h + 2M}^w \wedge n &\le \tau_h^s \wedge n, \quad L \le h, \\
  \label{eq:8} \tau_{h+2M}^s \wedge n& \le \tau^w_h\wedge n, \quad L+2M \le h.
\end{align}
\end{subequations}
\eqref{eq:1} is obvious since~they are both the infimum of the same expression
but $\tau^w_h$ is over a smaller range.
For~$\tilde{\tau}_h^w \le n$,
we have $\tau_{h+2M} \le [\tilde{\tau}_h^w]$ since
\begin{equation*}
  \big| \bar{W}^{(1)} \big( [\tilde{\tau}_h^w] \big) - \bar{W}^{(2)} \big( [\tilde{\tau}_h^w]\big) \big|
  \le  \big|\bar{W}^{(1)}(\tilde{\tau}_h^w)   - \bar{W}^{(2)}(\tilde{\tau}_h^w) \big| +2M = h + 2M,
\end{equation*}
where we use~\eqref{eq:BM-in-unit-interval} in the inequality,
and hence \eqref{eq:6} follows.
For~$\tau_h^s \le n$, we have
\begin{align*}
  \big| \bar{W}^{(1)}(\tau_h^s) - \bar{W}^{(2)}(\tau_h^s)  \big|
&  \le \big| \bar{S}^{(1)}(\tau_h^s) - \bar{S}^{(2)}(\tau_h^s)  \big| + 2M\\
 & = \big| S^{(1)}(\tau_h^s) - S^{(2)}(\tau_h^s)  \big| +2M \le h + 2M, 
\end{align*}
where the first inequality is due to~\eqref{eq:dist-bar-S-and-W-event} and the equality is due to~\eqref{eq:S-equal-S-bar}; so \eqref{eq:7} holds.
For~$\tau_h^w \le n$, we have
\begin{align*}
  \big|  S^{(1)}(\tau_h^w) - S^{(2)}(\tau_h^w)  \big|
&  = \big| \bar{S}^{(1)}(\tau_h^w)  - \bar{S}^{(2)}(\tau_h^w) \big| \\
&  \le \big| \bar{W}^{(1)}(\tau_h^w) - \bar{W}^{(2)}(\tau_h^w)  \big| +2M \le h+2M,
\end{align*}
where the equality follows from~\eqref{eq:S-equal-S-bar} and the fact that $\tau_h^w \le \tau_L^s$ for~$h \ge L+2M$  as a consequence of~\eqref{eq:7}, and the first inequality is due
to~\eqref{eq:dist-bar-S-and-W-event}; this gives~\eqref{eq:8}.

On the event~$\Lambda$, from~\eqref{eq:order-of-stopping-times} we have 
\begin{equation*}
\tau = \tau_{3M}^w \le \tau_L^s \wedge \tilde{\tau}_0^w, 
\end{equation*}
and \eqref{eq:before-coal-time-m-2} follows.
For~\eqref{eq:after-coal-time-m-2}, it suffices to show that
\begin{equation*}
  \tilde{\tau}^w_0 \le \tau_{3M}^w + n^{2\kappa-\eps}, \quad
  \tau^s_0 \le \tau_{3M}^w + n^{2\kappa-\eps}
\end{equation*}
holds with probability at least~$1-Cn^{-\delta}$.
 By~\eqref{eq:1} and Lemmas~\ref{lem:hitting-time-BM}, we have 
\begin{equation}\label{eq:tau-w-close-prob}
\begin{split}
  \Pp \big( \tilde{\tau}_0^w - \tau_{3M}^w \ge n^{2\kappa -\eps} \big)
  &  \le    \Pp \big( \tilde{\tau}_0^w - \tilde{\tau}_{3M}^w \ge n^{2\kappa -\eps} \big) \\
&  \le C \frac{3M}{n^{\kappa-\eps/2}}  +Cn^{-2}\le Cn^{-\delta}.
\end{split}
\end{equation}
By~\eqref{eq:8} and Lemmas~\ref{lem:hitting-time-BM}, we have 
\begin{equation}\label{eq:tau-s-close-prob}
\begin{split}
  \Pp \big( \tau_0^s - \tau_{3M}^w \ge n^{2\kappa -\eps} \big)
&  \le    \Pp \big( \tau_0^s - \tau_{5M}^s \ge n^{2\kappa -\eps} \big) + \Pp (\Lambda^c)\\
&  \le C \frac{3M}{n^{\kappa-\eps/2}} + Cn^{-\delta} \le Cn^{-\delta}.
\end{split}
\end{equation}

Next, we will show that~\eqref{eq:deviation-in-I-m-2} holds with probability at least~$1-Cn^{-\delta}$.
Since $\tau=\tau_{3M}^w$ is a stopping time for~$\WW=(W^{(1)},W^{(2)})$ and~$W^{(1)}$ is a Brownian motion,  by Lemma~\ref{lem:hitting-time-BM} and the strong Markov property we have
\begin{equation}\label{eq:deviation-of-w-prob}
\Pp \Big(  \sup_{\tau \le t \le \tau+n^{2\kappa - \eps}}|W^{(1)}_t - W^{(1)}_{\tau}| \ge n^{\kappa} \Big) \le C \frac{n^{\kappa-\frac{\eps}{2}}}{n^{\kappa}} e^{-n^{2\eps}/2} \le Cn^{-\delta}.
\end{equation}
On the other hand, we have~
\begin{equation*}
\tau^s_{5M} \le \tau^w_{3M} \le \tau^s_0 \wedge \tilde{\tau}^w_0
\end{equation*}
on~$\Lambda$ and~$\tau_{5M}^s$ is a stopping time for~$\Ss$; then by~\eqref{eq:tau-w-close-prob}, \eqref{eq:tau-s-close-prob}, the
strong Markov property and Lemma~\ref{lem:displacement-of-random-walk}
with~$p$ sufficiently close to~$\frac{1}{2}$, we have
\begin{equation}\label{eq:deviation-of-s-prob}
\begin{split}
&\quad  \Pp \Big( \max_{\tau \le k \le \tau+n^{2\kappa - \eps}}|S^{(1)}_k - S^{(1)}_{\tau}| \ge n^{\kappa}  \Big)\\
  &  \le Cn^{-\delta}
    + \Pp \Big( \max_{0\le k - \tau_{5M}^{s} \le 2n^{2\kappa - \eps}}|S^{(1)}_k - S^{(1)}(\tau_{5M}^s)| \ge \frac{n^{\kappa}}{2}  \Big) \le Cn^{-\delta}.
\end{split}
\end{equation}

Lastly, we will derive~\eqref{eq:19} from~\eqref{eq:dist-S-W-bar-small-2}, \eqref{eq:properties-of-I-m-2} and~\eqref{eq:deviation-in-I-m-2}.
For~$i=1$, \eqref{eq:19} follows from \eqref{eq:dist-S-W-bar-small-2} since~$\sigma=(2\ 1)$ implies~$S^{(2)} \equiv \bar{S}^{(2)}$ and $W^{(2)}\equiv \bar{W}^{(2)}$.
For~$i=2$, we will look at three ranges of~$k$: $k \le I$, $k \in I$ and~$k > I$.
For~$k \le I$,  \eqref{eq:19} follows from~\eqref{eq:before-coal-time-m-2} and~\eqref{eq:dist-S-W-bar-small-2}.
In particular, we have~$|W^{(1)}_{\tau}-S^{(1)}_{\tau}| \le M$, and hence for~$k \in I$, by~\eqref{eq:deviation-in-I-m-2} we have 
\begin{equation*}
|S^{(1)}_k - W^{(1)}_k | \le |W^{(1)}_{\tau}-S^{(1)}_{\tau}| + 2n^{\kappa} \le M+2n^{\kappa}.
\end{equation*}
For~$k \ge I$, by~\eqref{eq:after-coal-time-m-2} we have~$S^{(2)}_k = S^{(1)}_k$ and~$W^{(2)}_k = W^{(1)}_k$, so the inequality has already been established.
\end{proof}

The~$m=3$ case follows from the~$m=2$ case with an appropriate choice of~$\sigma$.
\begin{lemma}
\label{lem:m-equal-3}
Let~$m=3$ and~$\sigma = (2 \ 1 \ 3)$.  Then \eqref{eq:n-to-one-quater-dist-m-traj} holds for the processes~$\Ss$ and~$\WW$ constructed above.
\end{lemma}
\begin{proof}
    we have
  \begin{multline*}
\Pp \Big(  \max_{0 \le k \le n} \|\Ss_k - \WW_k\|_{\infty} > 3n^{\kappa} \Big)
\le \Pp \Big( \max_{0 \le k \le n} \|S^{(i)}_k - W^{(i)}_k\|_{\infty} > 3n^{\kappa}, \ i = 1, 2
\Big)\\
+ \Pp \Big( \max_{0 \le k \le n} \|S^{(i)}_k - W^{(i)}_k\|_{\infty} > 3n^{\kappa}, \ i = 2,3
\Big).
\end{multline*}
Due to our choice of~$\sigma$ and our construction, the event in the first probability is completely independent of~$S^{(3)}$, $\xi_k^{(3)}$ and~$U_3$.
As proved in~Lemma~\ref{lem:m-2-existence-of-Ii}, the first probability is bounded by~$Cn^{-\delta}$.
By symmetry the second probability is the same as the first one.
This proves~\eqref{eq:n-to-one-quater-dist-m-traj} for~$m=3$.
\end{proof}

For~$m \ge 4$, we will use induction to isolate the time intervals in which~\eqref{eq:follower-functions-agree} fails.
A na\"ive induction argument goes as follows.
Suppose that such intervals have been found for all~$i'$ with~$\sigma(i') < \sigma(i)$.
Initially, ~${g_i(0)=f_i(0)}$.
Let~$k_1$ be the first time that either~$S^{(i)}$ coalesces with~$S^{(j)}$, or~$W^{(i)}$ coalesces with~$W^{(j)}$, for some~$\sigma(j) < \sigma(i)$.
This creates a discrepancy~$f_i(k_1)\neq g_i(k_1)$.
But since~$S^{(i)}\approx W^{(i)}$ and~$S^{(j)} \approx W^{(j)}$ right before~$k_1$, the coalescence times of the two pairs~$(S^{(i)}, S^{(j)})$
and~$(W^{(i)}, W^{(j)})$ should also be~$O(n^{2\kappa-\eps})$ close to each other as seen from the proof of the~$m=2$ case.
Therefore, there is~$k_2>k_1$ such that~$f_i(k_2) = g_i(k_2)=j$, and then we can use the induction hypothesis.

However, there are two issues.
If~$g_i(k_0)=f_i(k_0) = j$ for some~$k_0$, letting~$k_1 >k_0$ be first time that a discrepancy~$g_i(k_1)\neq f_i(k_1)$ occurs, then there are two possibilities: 
\begin{enumerate}[1)]
\item there is some~$j_1$ with~$\sigma(j_1) < \sigma(j)$, such that 
\begin{equation}\label{eq:26}
g_i(k_1) = j_1, \ f_i(k_1) = j, \quad \text{ or } \quad g_i(k_1) = j, \ f_i(k_1) = j_1;
\end{equation}
\item there are~$j_1\neq j_2$ with~$\sigma(j_1), \sigma(j_2) < \sigma(j)$, such that 
\begin{equation}\label{eq:27}
g_i(k_1) = j_1, \quad f_i(k_1) = j_2.
\end{equation}
\end{enumerate}
The previous argument cannot handle the second possibility.
In the case~$m=3$ this is completely avoided \eqref{eq:27} using a specific choice of~$\sigma$, but in general it can only be ruled out on some event with high probability.

The second issue is that even~\eqref{eq:27} can be ruled out, the condition~$S^{(j)} \approx W^{(j)}$ cannot be guaranteed by the induction hypothesis.
We recall that the condition~${|S^{(j)} - W^{(j)}| = O(\log n)}$ is needed to get good bounds on the difference of coalescence time, but this is not valid when~\eqref{eq:follower-functions-agree} fails.
In other words, the discrepancy between the collision times of the $\mathbf{S}$ and $\mathbf{W}$ processes will propagate.  The root of this issue is the discontinuity of the mappings $\phi_{\sigma}$ and $\tilde{\phi}_{\sigma}$.

Our goal is to construct an event on which both issues can be resolved.
It turns out that the an important ingredient is certain separability of collisions times experienced by every single walker~$S^{(i)}$ and Brownian particle~$W^{(i)}$.

Let us start with a result on the coalescence times of three Brownian motions.
\begin{lemma}
\label{lem:collision-of-three-particles}
Let~$\BB = (B^{(-1)}, B^{(0)}, B^{(1)})$ be coalescing Brownian motions of three particles.
Assume that
\begin{equation*}
B_0^{(-1)} + 1 \le B^{(0)}, \quad B^{(0)}_0 + 1 \le  B^{(1)}_0.
\end{equation*}
Let $\tau_{\pm} = \inf \{  t \ge 0: B^{(0)}_t = B^{(\pm 1)}_t \}$.
Then for~$p \in (0, \frac{1}{2})$ and~$0<a<a_0(p)$, 
\begin{equation*}
  \Pp \big(  |\tau_- - \tau_+| \ge a \big) \ge 1 - C a^p,
\end{equation*}
where~$C=C_p$ is a constant depending on~$p$.
\end{lemma}
\begin{proof}
  Let~$\XX_t = (B^{(0)}_t - B^{(-1)}_t, B^{(1)}_t - B^{(0)}_t)$.  Then~$X_t$ is a two-dimensional Brownian motion in the first quadrant~$\R^2_{\ge 0} =\{ (x,y): x, y\ge 0 \}$, and until hitting the boundary it has constant diffusion  matrix~${A = \begin{bsmallmatrix} 2 & -1 \\ -1 & 2   \end{bsmallmatrix}}$.
  Once it hits the boundary of~$\R^2_{\ge 0}$, it will remain on it and perform Brownian motion with diffusion constant~$2$ until it hits the origin.
  The stopping times~$\tau_-$ and~$\tau_+$ are hitting times of the~$x$- and $y$-axis.

  Let~$(u,v) = \XX_0$.  We have~$u, v \ge 1$.
  We denote by~$\lambda_{u,v}$ be the hitting measure of the boundary of~$\R^2_{\ge 0}$, depending on the initial condition~$(u,v)$.  It is easy to see that, uniformly for~$u, v\ge 1$, we
  have~
\begin{equation}
\label{eq:16}
\lambda_{u,v} (\Gamma_{a,\eps}) \le Ca^{p}, 
\end{equation}
where 
\begin{equation*}
\Gamma_{a,\eps} =  \{(x,y): x=0, \ 0 \le y \le a^{\frac{1}{2} - \eps}, \text{ or } 0 \le x \le a^{p}, \ y = 0  \}.
\end{equation*}

By the strong Markov property, We have 
\begin{equation*}
  \Pp \big(  |\tau_- - \tau_+| \le a \big)
  \le 
  \Pp \big( \XX_{\tau_- \wedge \tau_+} \not\in \Gamma_{a,\eps} \big)
    + \Pp \big(  \tilde{\tau}_0 \le a \mid \tilde{B}_0 = a^p \big),
\end{equation*}
where~$\tilde{B}_t$ is a Brownian motion with diffusion constant~$2$ and~$\tilde{\tau}_0$ is its hitting time of~$0$.
The first term on the right hand side can be bounded by~\eqref{eq:16}, and by Lemma~\ref{lem:hitting-time-BM} we have
\begin{equation*}
  \Pp \big(  \tilde{\tau}_0 \le a \mid \tilde{B}_0 = a^{p} \big) \le Ca^{\frac{1}{2}-p} e^{-a^{2p-1}} \le Ca^{p}
\end{equation*}
for small enough~$a$.
This finishes the proof of the lemma.
\end{proof}

Let~$\sigma \in \Sigma_m$ and
\begin{equation}\label{eq:def-Di-Ei}
\begin{gathered}
  D_i = \{ \text{jump points of $g_i$} \} = \{ 1 \le k \le n: g_i (k) \neq g_i(k-1) \}, \\
  E_i = \{ \text{discontinuity points of~$f_i$} \}.
\end{gathered}
\end{equation}
The sets~$D_i$ and~$E_i$ depend not only on the processes~$\Ss$ and~$\WW$, but also on our choice of~$\sigma$.

Let~$\Ac_i=\Ac^{\sigma}_i$ be the set of indices that can be achieved by~$f_i$ (or~$g_i$).
These sets depend only on~$\sigma$: 
let
\begin{equation*}
  i_- = \max \{ j <i: \sigma(j) < \sigma(i) \}, \quad
  i_+ = \min \{ j > i: \sigma(j) < \sigma(i) \};
\end{equation*}
then  recursively we have
\begin{equation}\label{eq:recursive-relation-Ai}
\begin{split}
  \Ac_i &= \{ i \}, \quad i = \sigma^{-1}(1), \\
  \Ac_i &= \{ i \} \cup \Ac_{i_-}  \cup \Ac_{i_+} = 
\begin{cases}
  \{ i \} \cup \Ac_{i_-}, & \text{if $\sigma(i_-)>\sigma(i_+)$,}  \\
\{ i \} \cup \Ac_{i_+}, & \text{if $\sigma(i_-) <\sigma(i_+)$.}
\end{cases}
\end{split}
\end{equation}
Our convention is to set~$\sigma(i_{\pm}) = 0$ and~$\Ac_{\pm} = \varnothing$ if~$i_\pm$ does not exist.

As we will see in \eqref{eq:def-of-sigma}, we should choose~$\sigma$ to minimize $\max_i |\Ac_i|$, so that every Brownian particle/random walk will experience as few collisions as possible.

For~$h \ge 0$, let 
\begin{equation}\label{eq:def-of-tau-w-pm}
  \tilde{\tau}_{h,\pm}^{i,w} =
\begin{cases}
  \min \{  k \in \N: |W^{(i)}_k - W^{(i_{\pm})}_k| \le h   \},  & \text{if $i_{\pm}$ exists},  \\
\infty,   & \text{if $i_{\pm}$ does not exists}.
\end{cases}.
\end{equation}
In particular, $\tilde{\tau}^{i,w}_{0,\pm} = \inf \{ t: W^{(i_{\pm})}_t = W^{(i)}_t \}$ are the collision times of~$W^{(i)}$ and~$W^{(i_{\pm})}$.

\begin{lemma}
  \label{lem:tau-separated-with-higher-ranks}
For each~$i$, with probability at least $1 - C (|\Ac_i|-1)n^{-\delta}$,
\begin{equation}
  \label{eq:Ei-sigma-separated}
|  \tilde{\tau}_{0,+}^{i,w} - \tilde{\tau}_{0,-}^{i,w}| \ge 3n^{2\kappa-\eps},\quad 
  d \big( \tilde{\tau}_{0,+}^{i,w}, E_{i_+} \big) \ge 3n^{2\kappa-\eps}, \quad
      d \big( \tilde{\tau}_{0,-}^{i,w}, E_{i_-} \big) \ge 3n^{2\kappa-\eps},
    \end{equation}
    whenever~$i_{\pm}$ exists.
    Here, $d(x,F) = \min \{ |x-y|: y \in F \}$ and $E_i$ are introduced in~\eqref{eq:def-Di-Ei}.
\end{lemma}

\begin{proof}Let us assume that both~$i_{\pm}$ exist and ${\sigma(i_-) < \sigma(i_+)}$ 
  without loss of generality.
   By~\eqref{eq:recursive-relation-Ai}, the element in~$\Ac_i$ can be listed as
\begin{equation*}
i_{-r} < \cdots < i_{-1} = i_- <  i_0 =i < i_1 = i_+  <\cdots < i_s,
\end{equation*}
such that~$r+s = |\Ac_i| - 1$ and 
\begin{equation*}
  \sigma(i_{-r}) < \cdots < \sigma(i_{-1})  < \sigma(i_0), \quad
  \sigma(i_0) > \sigma(i_1) > \cdots > \sigma(i_s).
\end{equation*}
We also have~$ \Ac_i =   \{ i \} \cup \Ac_{i_+}$.

We claim that the event~\eqref{eq:Ei-sigma-separated} can be rewritten as
\begin{equation}\label{eq:15}
  |\tilde{\tau}_{0,-}^{i,w} -\tilde{\tau}_{0,+}^{i,w}| \ge 3n^{2\kappa-\eps}, \quad
  d(\tilde{\tau}_{0,-}^{i,w}, E_-) \ge 3n^{2\kappa - \eps}, \quad
    d(\tilde{\tau}_{0,+}^{i,w}, E_+) \ge 3n^{2\kappa - \eps},
\end{equation}
where
\begin{align*}
 E_{-} &=  \{ \text{collision times of~$W^{(i_-)}$ and~$W^{(i_j)}$, $j = -r, \ldots, -2$} \},\\
  E_{+} &= \{ \text{collision times of~$W^{(i_+)}$ and~$W^{(i_j)}$, $j = 2,\ldots,s$} \},
\end{align*}
Clearly, \eqref{eq:Ei-sigma-separated} implies \eqref{eq:15} since~$E_{\pm} \subset E_{i_{\pm}}$.
To see the other direction, let~$t \in E_{i_+}\setminus E_{+}$.
Then~$t$ is the coalescence time between~$W^{(i_+)}$ and~$W^{(i_j)}$ for some~$-r \le j \le -1$.
By monotonicity we must have~$t \ge \tilde{\tau}_{0,-}^{i,w} \ge \tilde{\tau}_{0,+}^{i,w}$, so
\begin{equation*}
|t-\tilde{\tau}_{0,+}^{i,w}| \ge |\tilde{\tau}_{0,-}^{i,w} - \tilde{\tau}_{0,+}^{i,w}| \ge 3n^{2\kappa-\eps}.
\end{equation*}
Similarly, if~$t \in E_{i_-}\setminus E_-$, then~we have~$t >\tilde{\tau}_{0,+}^{i,w} \ge \tilde{\tau}_{0,-}^{i,w}$ and hence
$|t-\tilde{\tau}_{0,-}^{i,w}| \ge 3n^{2\kappa-\eps}$.

  Applying a rescaled version of Lemma~\ref{lem:collision-of-three-particles} to~$(W^{(i_-)}, W^{(i)}, W^{(i_+)})$, we have
\begin{equation}\label{eq:17}
\Pp \big( |\tau- \tau'| < 3n^{2\kappa-\eps} \big) \le Cn^{-\delta}.
\end{equation}
A similar application of Lemma~\ref{lem:collision-of-three-particles} to
$(W^{(i_j)}, W^{(i_-)}, W^{(i)})$, $j = -r, \ldots, -2$, together with a union bound gives
\begin{equation}\label{eq:18}
  \Pp \Big( d\big( \tilde{\tau}^{i,w}_{0,-}, E_- \big) \le 3 n^{2\kappa - \eps } \Big) \le C(r-1)n^{-\delta}.
\end{equation}
Similarly,  we have 
\begin{equation}\label{eq:29}
  \Pp \Big( d \big( \tilde{\tau}^{i,w}_{0,+}, E_+\big) \le 3 n^{2\kappa - \eps } \Big) \le C(s-1)n^{-\delta}.
\end{equation}
The lemma follows from~\eqref{eq:17}, \eqref{eq:18} and~\eqref{eq:29}.
\end{proof}

For~$h \ge 0$, let 
\begin{equation}\label{eq:def-of-tau-s-pm}
  \tau_{h,\pm}^{i,s} =
\begin{cases}
  \min \{  k \in \N: |S^{(i)}_k - S^{(i_{\pm})}_k| \le h   \},  & \text{if $i_{\pm}$ exists},  \\
\infty,   & \text{if $i_{\pm}$ does not exists}.
\end{cases}.
\end{equation}
An analogous result to Lemma~\ref{lem:tau-separated-with-higher-ranks} also holds for the coalescing random walks.
\begin{lemma}
  \label{lem:tau-S-separated-with-higher-ranks}
For each~$i$, with probability at least $1 - C (|\Ac_i|-1)n^{-\delta}$,
\begin{equation}
  \label{eq:Di-sigma-separated}
  |\tau_{0,+}^{i,s} - \tau_{0,-}^{i,s}| \ge 3n^{2\kappa-\eps}, \quad 
  d \big( \tau_{0,+}^{i,s}, D_{i_+} \big) \ge 3n^{2\kappa-\eps}, \quad
      d \big( \tau_{0,-}^{i,s}, D_{i_-} \big) \ge 3n^{2\kappa-\eps},
    \end{equation}
    whenever~$i_{\pm}$ exists.  The sets $D_i$ are introduced in~\eqref{eq:def-Di-Ei}.
\end{lemma}

\begin{proof}
  From the proof of Lemma~\ref{lem:tau-separated-with-higher-ranks}, it suffices to prove the
  following three-trajectory estimate playing the role of Lemma~\ref{lem:collision-of-three-particles}: for any~$i_{-1} < i_0 < i_1$, 
\begin{equation*}
\Pp \big( |\tau_- - \tau_+| \le 3n^{2\kappa-\eps} \big) \le Cn^{-\delta},
\end{equation*}
where 
\begin{equation*}
\tau_{\pm} = \min\{ k: S^{(i_{\pm 1})}_k = S^{(i_0)}_k \}.
\end{equation*}

The proof of Lemma~\ref{lem:m-equal-3} implies that,  there exists a coupling between some coalescing Brownian motions~$(B^{(-1)}, B^{(0)}, B^{(1)})$ and the coalescing random walks~$(S^{(i_{-1})}, S^{(i_0)},
S^{(i_1)})$.
Let~$\tau_{\pm}'$ be the corresponding coalescing time for~$\BB$.
The proof of Lemma~\ref{lem:m-2-existence-of-Ii} implies
\begin{equation}\label{eq:30}
\Pp \big(  |\tau_{\pm}'-\tau_{\pm} | \le n^{2\kappa-\eps} \big) \ge 1 - Cn^{-\delta}.
\end{equation}
On the other hand, by Lemma~\ref{lem:collision-of-three-particles} we have 
\begin{equation}\label{eq:31}
\Pp \big( |\tau_-' - \tau_+'| \le 5n^{2\kappa-\eps} \big) \le Cn^{-\delta}.
\end{equation}
The lemma then follows from~\eqref{eq:30} and~\eqref{eq:31}.
\end{proof}

Recalling the random times~$\tilde{\tau}_{h,\pm}^{i,w}$ defined in~\eqref{eq:def-of-tau-w-pm} and~$\tau_{h,\pm}^{i,s}$ in~\eqref{eq:def-of-tau-s-pm}, we further define~$\tilde{\tau}_h^{i,w} =
\tilde{\tau}_{h,-}^{i,w} \wedge \tilde{\tau}_{h,+}^{i,w}$ and $\tau_h^{i,s} = \tau_{h,-}^{i,s} \wedge \tau_{h,+}^{i,s}$.
We note that this definition of~$\tau_h^{i,s}$ is consistent with~\eqref{eq:def-of-tau-s-i}.
Let us also define for~$h \ge 0$, 
\begin{align*}
\begin{split}
  \tau_{h,\pm}^{i,w} &= 
\begin{cases}
  \min \{  k \in \N: |W^{(i)}_k - W^{(i_{\pm})}_k| \le h   \},  & \text{if $i_{\pm}$ exists},  \\
\infty,   & \text{if $i_{\pm}$ does not exists}, 
\end{cases}\\
\end{split}
\end{align*}
and~$\tau_h^{i,w} = \tau_{h,-}^{i,w} \wedge \tau_{h,+}^{i,w}$.

The next lemma is a summary of our induction argument which uses the additional conditions~\eqref{eq:Ei-sigma-separated} and~\eqref{eq:Di-sigma-separated}.
\begin{lemma}
  \label{lem:existence-of-Ii-m-3}
  Assume~\eqref{eq:dist-bar-S-and-W-event}, \eqref{eq:BM-in-unit-interval}, \eqref{eq:Ei-sigma-separated}, \eqref{eq:Di-sigma-separated}  and the following conditions
  \begin{gather}
    \label{eq:closeness-of-tau-w}
    \tilde{\tau}_{0,\pm}^{i,w} \le  \tilde{\tau}_{7M, \pm}^{i,w} +n^{2\kappa - \eps}, \\
\label{eq:closeness-of-tau-s}    
\tau_{0,\pm}^{i,s} \le  \tau_{5M,\pm}^{i,s} + n^{2\kappa - \eps}, \\
\label{eq:deviation-of-S-in-I}
\max_{ 0 \le k - \tau_{5M}^{i,s} \le 2n^{2\kappa-\eps}} |S^{(i)}_k - S^{(i)}(\tau_{5M}^{i,s})| \le \frac{ n^{\kappa}}{2}, \\
\label{eq:deviation-of-W-in-I}
  \sup_{ 0 \le t - \tau_{3M}^{i,w} \le n^{2\kappa-\eps}} |W^{(i)}_k - W^{(i)}(\tau_{3M}^{i,w})| \le n^{\kappa},
\end{gather}
hold for all~$i$.
Then we have:
\begin{enumerate}[1)]
\item \label{item:order-of-random-times} The random times $\tau_h^{i,s}$, $\tilde{\tau}_h^{i,w}$, $\tau_h^{i,w}$ satisfy
  \begin{subequations}\label{eq:order-of-stopping-times-m-3}
  \begin{align}
\label{eq:22}    \tilde{\tau}_h^{i,w}  \le \tau_h^{i,w}, &\quad 0 \le h,\\ 
\label{eq:20}    \tau_{h+2M}^{i,w} \wedge n\le \tilde{\tau}_h^{i,w} \wedge n, &\quad 0 \le h, \\
\label{eq:23}    \tau_{h + 2M}^{i,w}\wedge n \le \tau_h^{i,s}\wedge n, &\quad L \le h \le 5M, \\
\label{eq:24}    \tau_{h+2M}^{i,s} \wedge n \le \tau^{i,w}_h  \wedge n, &\quad L+2M \le h \le 7M.
\end{align}
\end{subequations}

\item \label{item:structure-of-gi-fi}
  For each~$i$, the interval~$[0,n]$ can be decomposed into the union of the following closed intervals that have integer endpoints and intersect only at the boundary: 
\begin{equation}\label{eq:ordering-of-intervals}
0 \in R_{i,j_0} \le T_{j_0} \le R_{i, j_1} \le T_{j_1} \le R_{i,j_2} \le T_{j_2} \le \cdots \le n
\end{equation}
such that
\begin{equation}\label{eq:after-coal-time-m-3}
  \begin{split}
    g_i(k) = f_i(t) = j_q, \quad t, k \in R_{i,j_q},
\end{split}
\end{equation}
and each~$T_{j_q}$ contains exactly one jump point of~$g_i$ and one discontinuity point of~$f_i$, unless~$T_{j_q}$ is the last interval in~(\ref{eq:ordering-of-intervals}).

Moreover, the indices~$j_0, j_1, \ldots, j_p$ appearing in~(\ref{eq:ordering-of-intervals}) satisfy
\begin{equation*}
j_0 = i, \quad \sigma(j_0) > \sigma(j_1) > \ldots > \sigma(j_p).
\end{equation*}
The intervals~$T_j$ depend only on the index~$j$ and are  defined by
\begin{equation*}
T_j = \big[\tau_j, \tau_j + [n^{2\kappa - \eps} ] \big] \cap [0,n], \quad \tau_j = \tau_{3M}^{j,w}.
\end{equation*}
If the last interval in~(\ref{eq:ordering-of-intervals}) is~$R_{i,j_p}$, then necessarily~$T_{j_p} = \varnothing$.
\item \label{item:bounds} We have the bound
\begin{equation}
  \label{eq:derived-bounds-on-dist}
  |S^{(i)}_k - W^{(i)}_k| \le
\begin{cases}
M, & k \in  \bigcup_q R_{i, j_q},  \\
M+2n^{\kappa},   & k \in \bigcup_q T_{j_q}.
\end{cases}
\end{equation}
\end{enumerate}

\end{lemma}

\begin{proof}
  \begin{equation}\label{eq:Ti-contain-coal-time}
\tau \le n \quad\Rightarrow \quad T_i \cap [0,n] \neq \varnothing \quad \Rightarrow \quad \tau \in T_i.
\end{equation}
We will prove by induction on~$\sigma^{-1}(i)$.
  If~$\sigma(i) = 1$, then by definition, 
\begin{equation}\label{eq:32}
\tau_h^{i,s} = \tilde{\tau}_h^{i,w} = \tau_h^{i,w} = \infty, \quad h \ge 0.
\end{equation}
Thus~\eqref{eq:order-of-stopping-times-m-3} is obviously valid.
For item~\ref{item:structure-of-gi-fi}, $\sigma(i) = 1$ implies
\begin{equation}\label{eq:33}
S^{(i)}_k \equiv \bar{S}^{(i)}_k, \quad W^{(i)}_t \equiv \bar{W}^{(i)}_t, \quad 0 \le t, k \le n.
\end{equation}
So we have just one interval~$R_{i,i} = [0,n]$ in~(\ref{eq:ordering-of-intervals}), and by~(\ref{eq:32})~$T_i = \varnothing$.
Item~\ref{item:bounds} immediately follows from~(\ref{eq:33}) and~\eqref{eq:dist-bar-S-and-W-event}.

Now let us assume that items~\ref{item:order-of-random-times}-~\ref{item:bounds} have already been established for all~$i'$ with~$\sigma(i') < \sigma(i)$.
We will prove them for~$i$.

\smallskip

\noindent \textbf{Part 1: proof of item~\ref{item:order-of-random-times}.}

The first two inequalities~\eqref{eq:22} and~\eqref{eq:20} do not rely on the the induction hypothesis.
In fact, \eqref{eq:22} is obvious from the definition.
To see~\eqref{eq:20}, we claim that under~(\ref{eq:Ei-sigma-separated}) and~\eqref{eq:BM-in-unit-interval}, we have
\begin{equation}\label{eq:cBM-in-unit-interval}
    \sup_{t \in [k,k+1]} |W^{(i)}_t - W^{(i)}_k| \le 3 \log n \le M, \qquad k=0,\ldots,n-1, \quad i = 1, \ldots, m.
  \end{equation}
  
Fixing~$i$ and~$k$, if~$[k,k+1]$ does not contain any discontinuity point of~$f_i(\cdot)$, then~\eqref{eq:cBM-in-unit-interval} follows from~\eqref{eq:BM-in-unit-interval} since for~$j=f_i([k,k+1])$, 
\begin{equation*}
W^{(i)}_t = \bar{W}^{(j)}_t, \quad t \in [k,k+1].
\end{equation*}
Otherwise, by \eqref{eq:Ei-sigma-separated} there is at most one point~$t_0 \in [k,k+1]$ at which~$f_i(\cdot)$ is discontinuous.
Writing~$j_1 = f_i(k)$ and~$j_2=f_i(k+1)$, we have 
\begin{equation*}
  |W^{(i)}_t - W^{(i)}_k| \le \sup_{t' \in [t_0, k+1]} |\bar{W}^{(j_2)}_{t'} - \bar{W}^{(j_2)}_{t_0}|
  + \sup_{t' \in [k,t_0]} |\bar{W}^{(j_1)}_{t'} - \bar{W}^{(j_1)}_k|.
\end{equation*}
This proves the claim.

To prove~\eqref{eq:20}, assume~$\tilde{\tau}_h^{i,w}<n$ (otherwise there is nothing to prove). Then for~$j=i_-$ or~$i_+$, we have
\begin{equation*}
|W^{(i)}(\tilde{\tau}_h^{i,w})-W^{(j)}(\tilde{\tau}_h^{i,w})| \le h.
\end{equation*}
Using~\eqref{eq:cBM-in-unit-interval}, we have 
\begin{equation*}
\big| W^{(i)} \big( [\tilde{\tau}_h^w] \big) - W^{(j)} \big( [\tilde{\tau}_h^w]\big) \big|
  \le  \big|W^{(i)}(\tilde{\tau}_h^w)   - W^{(j)}(\tilde{\tau}_h^w) \big| +2M \le h + 2M,
\end{equation*}
and this implies~\eqref{eq:20}.

\smallskip

Let us look at~\eqref{eq:23}.
Assuming~$\tau_h^{i,s} < n$ and let~$j = i_{\pm}$ be such that 
\begin{equation*}
|S^{(i)}(\tau_h^{i,s}) - S^{(j)}(\tau_h^{i,s})| \le h.
\end{equation*}
By induction hypothesis, $[0,n]$ can be decomposed into the union of intervals 
\begin{equation}\label{eq:34}
0 \le R_{j, r_0} \le T_{r_0} \le R_{j, r_1} \le T_{r_1} \le \cdots \le n, 
\end{equation}
with~$r_0 = j$, such that item~\ref{item:structure-of-gi-fi} holds for~$j$.
We claim that
\begin{equation}
\label{eq:21}
\tau_h^{i,s} \not\in  \bigcup_{r_l}T_{r_l}.
\end{equation}
Otherwise, assume the contrary that~$\tau_h^{i,s} \in T_{j'}$ for some~$T_{j'}$ appearing in~(\ref{eq:34}).
We have 
\begin{equation*}
\tau_{j'} = \tau^{j',w}_{3M} \le \tau_h^{i,s} \le \tau_{j'} + n^{2\kappa-\eps}.
\end{equation*}
Then, by~(\ref{eq:closeness-of-tau-s}) and $h \le 5M$,  we have 
\begin{equation*}
\tau_{j'} \le \tau_0^{i,s} \le \tau_{j'} + 2n^{2\kappa-\eps}.
\end{equation*}
We will show that~
\begin{equation}\label{eq:35}
\tau_{j'} \le \tau_0^{j',s} \le \tau_{j'} + n^{2\kappa-\eps},
\end{equation}
and hence~$|\tau_0^{i,s} - \tau_0^{j',s}| \le 2n^{2\kappa-\eps}$;
this contradicts with~(\ref{eq:Di-sigma-separated}) since~$\tau_0^{j',s} \in D_j$.
If~$T_{j'} \le n$, then~(\ref{eq:35}) is trivial; if not, we still have
\begin{equation*}
\tau_{5M}^{j',s} \le \tau_{3M}^{j',w} = \tau_{j'} \le \tau_0^{j',s}
\end{equation*}
by (\ref{eq:24}) and~(\ref{eq:23}).
And thus~(\ref{eq:35}) follows from~(\ref{eq:closeness-of-tau-s}).

Now that~\eqref{eq:21} is true, noting that~$S^{(i)}_k = \bar{S}^{(i)}_k$ for all~$k \le \tau_h^{i,s} \le \tau_L^{i,s}$, by~\eqref{eq:derived-bounds-on-dist} and~\eqref{eq:dist-bar-S-and-W-event} we have 
\begin{align*} 
  \big| \bar{W}^{(i)}(\tau_h^{i,s}) - W^{(j)}(\tau_h^{i,s})  \big|
&  \le \big| \bar{S}^{(i)}(\tau_h^{i,s}) - S^{(j)}(\tau_h^{i,s})  \big| + 2M\\
 & = \big| S^{(i)}(\tau_h^s) - S^{(j)}(\tau_h^s)  \big| +2M \le h + 2M.
\end{align*}
If~$\tau_h^{i,s} \le \tilde{\tau}_0^{i,w}$, then~$\bar{W}^{(i)}(\tau_h^{i,s})=W^{(i)}(\tau_h^{i,s})$ and~\eqref{eq:23} is proved; otherwise, \eqref{eq:23} follows from the definition
of~$\tau_h^{i,w}$ and~$\tilde{\tau}_h^{i,w}$ since
\begin{equation*}
\tau^{i,s}_h \ge \lceil \tilde{\tau}_0^{i,w} \rceil \ge\tau_{h'}^{i,w}, \quad h'\ge 0.
\end{equation*}

Lastly, we look at~\eqref{eq:24}.
Again assuming that~$\tau_h^{i,w}<n$ and let~$j=i_{\pm}$ be such that 
\begin{equation}\label{eq:25}
  |W^{(i)}(\tau_h^{i,w})  - W^{(j)}(\tau_h^{i,w})|  \le h, 
\end{equation}
and assume the decomposition~(\ref{eq:34}).
We can show that~$\tau_h^{i,w} \not\in \bigcup_{l } T_{r_l}$, and the proof is similar to that of~\eqref{eq:21}, except that we will use~(\ref{eq:Ei-sigma-separated}) to draw the contradiction instead
of~(\ref{eq:Di-sigma-separated}). 
Therefore, using~\eqref{eq:derived-bounds-on-dist} for~$j$ and~\eqref{eq:dist-bar-S-and-W-event}, 
we have
\begin{align*}
  \big|  S^{(i)}(\tau_h^{i,w}) - S^{(j)}(\tau_h^{i,w})  \big|
&  = \big| \bar{S}^{(i)}(\tau_h^{i,w})  - S^{(j)}(\tau_h^{i,w}) \big| \\
&  \le \big| \bar{W}^{(i)}(\tau_h^{i,w}) - W^{(j)}(\tau_h^{i,w})  \big| +2M \le h+2M,
\end{align*}
where the first line is due to~$\tau_h^{i,w} \le \tau_L^{i,s}$ for~$L+2M \le h \le 7M$ as a consequence of~\eqref{eq:23},
and the last inequality follows from~\eqref{eq:25} and~$W^{(i)}(\tau_h^{i,w}) = \bar{W}^{(i)}(\tau_h^{i,w})$, since~$\tau_h^{i,w} \le \tilde{\tau}^{i,w}_0$.

\smallskip

\noindent \textbf{Proof of item~\ref{item:structure-of-gi-fi}.}

If~$\tau_i =\tau_{3M}^{i,s}> n$, then~$T_i = \varnothing$.
On the other hand, (\ref{eq:order-of-stopping-times-m-3}) implies 
\begin{equation*}
\tau_L^{i,s}, \tilde{\tau}_0^{i,w} \ge \tau_{3M}^{i,w} \wedge n  \ge n,
\end{equation*}
and hence~(\ref{eq:33}) holds.
So we will have just one interval~$R_{i,i} = [0,n]$ in~(\ref{eq:ordering-of-intervals}).

Similarly, if~$\tau_i \le n \le \tau_i + [n^{2\kappa-\eps}]$, then (\ref{eq:ordering-of-intervals}) holds with~$R_{i,i} = [0, \tau_i]$ and~$T_i = [\tau_i, n]$.

If~$\tau_i + [n^{2\kappa-\eps}] < n$, then~$T_i < n$.
By~(\ref{eq:order-of-stopping-times-m-3}) we have~$\tilde{\tau}_0^{i,w} \in T_i$.
Without loss of generality, let us assume~$\tilde{\tau}_0^{i,w} = \tilde{\tau}_{0,-}^{i,w}$.
We first claim that~$\tau_i = \tau_{3M,-}^{i,w}$ and~$\tau_0^{i,s} = \tau_{0,-}^{i,s}$.
In fact, the argument in proving~\eqref{eq:23} and~\eqref{eq:25} implies
\begin{equation*}
\tau_{5M,-}^{i,s} \le \tau_{3M,-}^{i,w}, \quad \tilde{\tau}_{3M,-}^{i,w}  \le \tau_{3M,-}^{i,w}.
\end{equation*}
and thus by~\eqref{eq:closeness-of-tau-s} and~\eqref{eq:closeness-of-tau-w}, we have~$\tau_{0,-}^{i,s} \in T_i$ and~$\tilde{\tau}_{0,-}^{i,w} \in T_i$.
Then, as the length of~$T_i$ is less than~$n^{2\kappa-\eps}$, \eqref{eq:Ei-sigma-separated} and~\eqref{eq:Di-sigma-separated} imply that~$\tau_{0,+}^{i,s} \not\in T_i$ and~$\tilde{\tau}_{0,+}^{i,w}
\in T_i$.
This proves the claim.

Writing~$j = i_-$, by induction hypothesis for~$\sigma(j) < \sigma(i)$, we can decompose~$[0,n]$ into intervals of the form~\eqref{eq:34}.
Since each~$T_{r_l}$ contains at least one discontinuity point of~$f_j$ and~$\tilde{\tau}_{0,-}^{i, w} \in T_i$ , by~\eqref{eq:Di-sigma-separated}~$T_i \cap T_{r_l} = \varnothing$ for any$r_l$, and
hence for some~$j'=r_l$, we have
\begin{equation*}
T_{j'} < T_i \subset R_{j,j'} \le T_{r_{l+1}} \le \cdots \le n.
\end{equation*}
Moreover, 
\begin{equation*}
g_i(k) = g_{j'}(k) , \quad f_i(t) = f_{j'}(t), \quad T_i \le t, k \le n.
\end{equation*}
Therefore, the decomposition~(\ref{eq:ordering-of-intervals}) for~$i$ takes the form 
\begin{equation*}
 R_{i,i} \le T_i \le R_{i, j'} \le T_{r_{l+1}} \le \cdots  \le n, 
\end{equation*}
where~$R_{i,i} = [ 0, \tau_i]$, $R_{i, j'} = R_{j,j'} \setminus [0, \tau_i + [n^{2\kappa-\eps}]]$, and the rest of the chain is the same as the part after~$R_{j,j'}$ in~(\ref{eq:34}).
It is easy to check~(\ref{eq:after-coal-time-m-3}).

\smallskip

\noindent \textbf{Proof of item~\ref{item:bounds}.}

If~$k \in R_{i,i}$, then~$k \le \tau_i\le \tau_0^{i,s} \wedge \tilde{\tau}_0^{i,w}$.
Hence we have
\begin{equation*}\label{eq:14}
S_k^{(i)} = \bar{S}^{(i)}_k,  \quad W^{(i)}_k =  \bar{W}^{(i)}_k, 
\end{equation*}
and the desired bound follows from~\eqref{eq:dist-bar-S-and-W-event}.
In particular, for~$k = \tau_i$ we have 
\begin{equation*}
|W^{(i)}_k - S_k^{(i)}| \le M.
\end{equation*}
Then, for $k \in T_i $, since~
\begin{equation*}
\tau^{i,s}_{5M} \le \tau_i \le \tau_0^{i,s} \le \tau^{i,s}_{5M} + n^{2\kappa-\eps}, 
\end{equation*}
from~\eqref{eq:deviation-of-S-in-I} and~\eqref{eq:derived-bounds-on-dist} we have 
\begin{equation*}
  |W^{(i)}_k - S_k^{(i)}|
  \le |W^{(i)}_{\tau_i} - S^{(i)}_{\tau_i}| + |W^{(i)}_k-W^{(i)}_{\tau_i}|
  + |S^{(i)}_k - S^{(i)}_{\tau_i}|
  \le M+2n^{\kappa}
\end{equation*}
as desired.
If~$T_i < k$, the desired bound follows from the induction hypothesis.
\end{proof}

In light of Lemmas~\ref{lem:tau-separated-with-higher-ranks} and~\ref{lem:tau-S-separated-with-higher-ranks}, we should minimize~$\max_i |\mathcal{A}_i|$.
We define~$\sigma$ using a binary tree structure:
\begin{equation}\label{eq:def-of-sigma}
  \begin{split}
       \sigma(i) > \sigma(j) \Leftrightarrow
    &\  i=(2i'+1) 2^q, \ j = (2j'+1)2^p,\ q > p, \text{ or,} \\
    &\ i=(2i'+1)2^p, \ j = (2j'+1)2^p, \ i'>j'.
  \end{split}
\end{equation}
Let~$d = [\log_2 m]+1$.  By induction it is easy to see that 
\begin{equation*}
i = (2i'+1)2^p \Rightarrow |\mathcal{A}_i| = d-p.
\end{equation*}
Therefore, for~$\sigma$ defined in~\eqref{eq:def-of-sigma}, we have 
\begin{equation}
\label{eq:bound-on-Ai}
|\mathcal{A}_i| \le \log_2 m + 1, \quad i=1,\ldots,m, 
\end{equation}

We are ready to give the proof of Lemma~\ref{lem:key-coupling-lemma-finite-range} for~$m \ge 4$.

\begin{proof}[Proof of Lemma~\ref{lem:key-coupling-lemma-finite-range}]
  We choose~$\sigma \in \Sigma_m$ as in (\ref{eq:def-of-sigma}).
  We want to show that the processes~$\Ss$ and~$\WW$ defined at the beginning of this section satisfy (\ref{eq:n-to-one-quater-dist-m-traj}).
  It suffices to show that the assumptions in Lemma~\ref{lem:existence-of-Ii-m-3} hold with 
  probability at least~$1 - Cmn^{-\delta} \ln m$.
  Noting Lemmas~\ref{lem:dist-bar-S-and-W}, \ref{lem:BM-in-unit-interval}, \ref{lem:tau-separated-with-higher-ranks} and~\ref{lem:tau-S-separated-with-higher-ranks}, and 
\begin{equation*}
|\Ac_i| \le \log_2 m + 1, \quad i = 1, \ldots, m,
\end{equation*}
for our choice of~$\sigma$,
it remains to estimate the probabilities of the events given by~(\ref{eq:closeness-of-tau-w})-~(\ref{eq:deviation-of-W-in-I}).

For fixed~$i$, each of these events only involves two trajectories of the coalescing random walks, $S^{(i)}$ and~$S^{(i_{\pm})}$, or two trajectories of the coalescing Brownian motions,
$W^{(i)}$ and~$W^{(i_{\pm})}$.
Hence they are essentially the same kinds of events whose probabilities have
been estimated in~\eqref{eq:tau-w-close-prob}--\eqref{eq:deviation-of-s-prob} in the proof of Lemma~\ref{lem:m-2-existence-of-Ii},
and their probabilities are all at least~$1-Cn^{-\delta}$.
The desired conclusion then follows from a union bound for all~$i$.
\end{proof}
\subsection{General case}
In this section we will prove Lemma~\ref{lem:key-coupling-lemma} under the general conditions~\ref{it:higher-moments}-~\ref{it:decoupling}.

\begin{lemma}
\label{lem:one-step-decoupling}
Assume~\ref{it:decoupling}. Let $x_1< \ldots <x_m$ and $d = \min_i (x_i-x_{i-1})$.
Let $Q_{x_1,...,x_m}$ be the law of the random vector~$(\psi_1(x_1),\ldots,\psi_1(x_m))$.
Then 
\begin{equation*}
\| Q_{x_1,\ldots,x_m} - \mu^{\times m} \|_{TV} \le Cm d^{-\beta}.
\end{equation*}
\end{lemma}

\begin{proof}
Let~$Q_m = Q_{x_1, \ldots, x_m}$. 
  We have 
\begin{align*}
  \| Q_m - \mu^{\times m} \|_{TV}
    &\le \sum_{i = 2}^{m} \| Q_i\times\mu^{ \times (m-i)} - Q_{i-1}\times \mu^{\times (m-i+1)} \|_{TV} \\
    &= \sum_{i = 2}^{m} \| Q_i - Q_{i-1}\times \mu \|_{TV}.
\end{align*}
It suffices to show that~$\|Q_i - Q_{i-1}\times \mu\|_{TV} \le \phi(d)$.
This follows from Berbee's Lemma (see for example \cite[Theorem 3.1]{BryAPPROXIMATIONTHEOREMBERKES1982}) and our assumption~\ref{it:decoupling}.
\end{proof}

\begin{lemma}
\label{lem:coupling-S-S-bar}
Assume~\ref{it:decoupling}.  Let~$n, m \in \N$, $L > 0$ and~$\sigma \in \Sigma_m$.
Let $\xx \in \R^m$ such that~$x_{i-1}+L < x_i$.  Let~$d = \min_i (x_i-x_{i-1})$.
There is a coupling of the processes~$\Ss$, $\bar{\Ss}$ on~$\R^m$ such that: 
\begin{enumerate}[1)]
\item $\Ss=(\Ss_k)_{k \in \N}$ is a Markov chain with transition probability~\eqref{eq:distr-of-coal-RW-m-traj} and~$\Ss_0 = \xx$.
\item Let~$\bar{S}^{(i)}$ be the $i$-th component of $\bar{\Ss}$.
  Then~$\bar{S}^{(i)}$ are independent random walks that has step distribution~$\mu$ and~$\bar{S}^{(i)}_0 = x_i$.
\item With probability at least $1 - Cmnd^{-\beta}$, the following is true: 
\begin{equation}\label{eq:38}
  \bar{S}^{(i)}_k = S^{(i)}_k, \quad 0 \le k \le \tau^{i,s}_L \wedge n, \ 1 \le i \le m,
\end{equation}
where for $L \ge 0$,
\begin{equation}
\tau^{i,s}_L = \min \{ k: \exists j, \text{ s.t. } \sigma(j) < \sigma(i), \ |S^{(i)}_k - S^{(j)}_k |  \le  L \}.\label{eq:36}
\end{equation}
\end{enumerate}
Moreover, the coupling can be expressed as~$\bar{\Ss} = \Theta(\Ss, U)$ where~$U$ is a $\mathrm{U}[0,1]$-r.v.'s that is independent of~$\Ss$, and~$\Theta: (\R^m)^{\N}  \times [0,1] \to (\R^m)^{\N}$ is a deterministic functional.
\end{lemma}

\begin{proof}
  Suppose that on a common probability space there is
\begin{enumerate}[1)]
\item   a  Markov chain~${\Ss = (\Ss_k)_{k \in \N}}$ with transition probability~\eqref{eq:distr-of-coal-RW-m-traj} and~$\Ss_0 = \xx$,
\item     random variables~$\xi_k^{(i)}$, $1 \le i \le m$, $1\le k$ that are \iid and distributed as~$\mu$,
\item random variables~$\eta_k$, $1 \le k \le m$,  that are \iid and has distribution $\mathrm{U}[0,1]$,
\end{enumerate}
such that~$\Ss$ is independent of~$\xi^{(i)}_k$ and~$\eta_k$.
We will construct~$\bar{\mathbf{\Ss}}$ from~$\mathbf{\Ss}$, $\xi^{(i)}_k$ and~$\eta_k$.

By Lemma~\ref{lem:one-step-decoupling}, there are functions~$q_l:\R^l \times [0,1] \to \R^l$ such that, if~$\yy \in \R^l$ satisfies~$y_{i+1} -y_i \ge d$, and $\mathbf{Y} \sim Q_{y_1, \ldots, y_l}$
and~$\eta$ be~$\mathrm{U}[0,1]$-r.v.\ that is independent of~$\mathbf{Y}$, 
then $\mathbf{\bar{Y}} = q_l(\mathbf{Y}, \eta)$ will have distribution~$\mu^{\times l}$, and 
\begin{equation}\label{eq:37}
  \Pp (\mathbf{\bar{Y}} \neq \mathbf{Y}) \le Cld^{-\beta}.
\end{equation}

We notice that~$\tau_L^{i,s}$ defined in~(\ref{eq:36}) are stopping times for the Markov chain~$\Ss$.
Let us construct~$\bar{\Ss}$ as a Markov chain as follows.
We set~$\bar{\Ss}_0 = \Ss_0=\xx$.
For~$k \ge 0$, suppose that~$\mathbf{\Ss}$ has been constructed up to time~$k$, we will construct~$\mathbf{\Ss}_{k+1}$.
Let $  I_k $  be the set of indices~$i$ such that~$k < \tau^{i,s}_L$, whose elements can be listed as
\begin{equation*}
i_{k, 1} < i_{k,2} < \ldots < i_{k, r_k}.
\end{equation*}
We have from the definition,
\begin{equation*}
|S^{(i_{k,j})} - S^{(i_{k,j-1})} | \ge L, \quad 2 \le j \le r_k.
\end{equation*}
We set 
\begin{gather*}
  \big( \bar{S}^{(i_{k,1})}_{k+1}, \ldots, \bar{S}^{(i_{k,r_k})}_{k+1} \big)
  =   \big( \bar{S}^{(i_{k,1})}_k, \ldots, \bar{S}^{(i_{k,r_k})}_k \big) + q_{r_k} \Big( \big(  \Delta S^{(i_{k,1})}_k, \ldots,\Delta S^{(i_{k,r_k})}_k \big), \eta_k  \Big),\\
  \bar{S}^{(i)}_{k+1} = \bar{S}^{(i)}_k + \xi^{(i)}_{k - \tau^{i,s}_L}, \quad i \not\in I_k.
\end{gather*}
The index sets~$I_k$ are random, but they depend only on~$\tau^{i,s}_L$, which are stopping times of the Markov process~$\Ss$.
The strong Markov property then imply that~$\bar{\Ss}^{(i)}$ are independent random walks with step distribution~$\mu$.
By (\ref{eq:37}), for each~$k$, with probability at least~$1 - Cmd^{-\beta}$, we have 
\begin{equation*}
  \bar{S}_{k+1}^{(i)} -  \bar{S}^{(i)}_k =   S_{k+1}^{(i)} -  S^{(i)}_k
\end{equation*}
whenever~$k < \tau^{i,s}_L$.
It follows from a union bound that~(\ref{eq:38}) holds with probability at least~$1-Cmnd^{-\beta}$.
\end{proof}

\begin{proof}[Proof of Lemma~\ref{lem:key-coupling-lemma} under general conditions]
  We will essentially repeat the proof of Lemma~\ref{lem:key-coupling-lemma-finite-range}, with the following adjustment.
  First, we will use Lemma~\ref{lem:coupling-S-S-bar} to construct the independent random walks~$\bar{\Ss}$ from coalescing random walks~$\Ss$ with~$L = n^{\gamma}$, and~(\ref{eq:S-equal-S-bar}) holds
  with probability at least~$1-Cn^{-\beta\gamma+1}m$, while in previous setting, $L$ was the range of dependence, and~(\ref{eq:S-equal-S-bar}) holds with probability one.
  Second, we will couple the independent random walks~$\bar{\Ss}$ with the independent Brownian motions~$\bar{\WW}$ using the general form of Lemma~\ref{lem:cor-of-KMT}.
  Thus, by~(\ref{eq:n-gamma-close}), we have 
\begin{equation}
\label{eq:39}
\Pp \big( \max_{0 \le k \le n} \| \bar{\WW}_k - \bar{\Ss}_k \|_{\infty} \le n^{-\gamma}\big)
\ge 1 - Cn^{-\alpha\gamma+1}m.
\end{equation}
We will also have~$M = n^{\gamma}$ and
\begin{equation*}
\eps < \frac{1}{10} \Big( \min (\gamma\alpha-1, \gamma\beta-1, \kappa-\gamma,\theta-\kappa) - \delta \Big).
\end{equation*}

With these adjustments, Lemma~\ref{lem:existence-of-Ii-m-3} still holds, since the only connection between~$\Ss$ and~$\bar{\Ss}$ that we have used in the proof was~(\ref{eq:S-equal-S-bar}).
It remains to estimate the probabilities of the events~\eqref{eq:closeness-of-tau-w}--\eqref{eq:deviation-of-W-in-I} with our adjustment of~$M$ and the constants~$\delta,\kappa,\gamma,\theta,\delta$.
For~\eqref{eq:closeness-of-tau-w}, by Lemma~\ref{lem:hitting-time-BM} we have
\begin{equation*}
  \Pp \big( \tilde{\tau}_{0,\pm}^{i,w} - \tilde{\tau}_{7M, \pm}^{i,w}  \ge n^{2\kappa - \eps}
  \big) \le C \frac{n^{\gamma}}{n^{\kappa-\frac{\eps}{2}}} \le Cn^{-\delta}.
\end{equation*}
For~\eqref{eq:closeness-of-tau-s}, by Lemma~\ref{lem:passage-time-general-MC-1} with~$p,p_0$ sufficiently close to~$\frac{1}{2}$, we have 
\begin{equation*}
  \Pp \big(\tau_{0,\pm}^{i,s} -  \tau_{5M,\pm}^{i,s} \ge  n^{2\kappa - \eps}
  \big)
  \le C \frac{n^{2p_0\gamma}}{n^{p(2\kappa-\eps)}} \le Cn^{-\delta}.
\end{equation*}
For~\eqref{eq:deviation-of-S-in-I}, by Lemma~\ref{lem:displacement-of-random-walk} with~$p$ close to~$1$,  we have
\begin{equation*}
  \Pp \Big( \max_{ 0 \le k - \tau_{5M}^{i,s} \le 2n^{2\kappa-\eps}} |S^{(i)}_k - S^{(i)}(\tau_{5M}^{i,s})| \ge \frac{ n^{\kappa}}{2} \Big)
  \le C \frac{n^{\kappa-\frac{\eps}{2}}}{n^{\kappa}} + \frac{C}{n^{p(2\kappa-\eps)}} \le Cn^{-\delta}.
\end{equation*}
For~\eqref{eq:deviation-of-W-in-I}, by Lemma~\ref{lem:hitting-time-BM} we have
\begin{equation*}
  \Pp \Big( \max_{ 0 \le k - \tau_{3M}^{i,w} \le n^{2\kappa-\eps}} |W^{(i)}_k - W^{(i)}(\tau_{3M}^{i,w})|\ge n^{\kappa}\Big)
\le C \frac{n^{\kappa-\frac{\eps}{2}}}{n^{\kappa}} e^{-\frac{n^{\eps}}{2}} \le Cn^{-\delta}.
\end{equation*}
\end{proof}

\section{Appendix}\label{sec:appendix}
The first lemma is a standard estimate on the passage times of Brownian motions.
\begin{lemma}
\label{lem:hitting-time-BM}
Let~$(B_t)$ be the Brownian motion starting from $0$ with diffusion constant~$\nu$, i.e., $\E B_t^2 = \nu t$.
Let~$\tau_a = \inf \{ t \ge 0: B_t = a \}$ be the hitting time of~$a$, and let~$B_t^{*} = \sup_{0 \le s \le t} B_t$ be the maximal process.
Then for~$a, T > 0$,
\begin{gather}
\label{eq:passage-time-not-large}
  \Pp \big( \tau_a \ge T \big) = \Pp (|\Nc(0, T) | \le a ) \le \sqrt{\frac{2}{\pi}} \cdot \frac{a}{ \sqrt{T}} \\
\label{eq:passage-time-not-small}
\Pp  \big( B^{*}_T \ge a   \big) = \Pp \big(  \tau_a \le T \big) \le \sqrt{ \frac{2}{\pi}} \frac{\sqrt{T}}{a} e^{-\frac{a^2}{2T}}.
\end{gather}
\end{lemma}

The next lemma provides an estimate on the coalescing time of our monotone maps.
\begin{lemma}
  \label{lem:passage-time-general-MC-1}
  Let~$(X_n)$ be the Markov chain on~$[0,+\infty)$ with transition probaiblity
\begin{equation}
\label{eq:gap-markov-chain}
\Pp (X_{n+1} \in dy | X_n = x) = \Pp ( \Psi_{1,0}(x) - \Psi_{1,0}(0) \in dy).
\end{equation}
Assume~\ref{it:minimum-gap}-~\ref{it:decoupling}.
Then for any~$0 < p < p_0 < \frac{1}{2}$, there exists constant~$C=C(p, p_0)> 0$ such that 
\begin{equation}\label{eq:passage-time-not-large-MC}
\Pp (\tau_0 > T | X_0 = x ) \le C \frac{x^{2p_0}}{T^p}, \quad x \ge 1,
\end{equation}
where~$\tau_h = \min \{ n: X_n \le h \}$ for $h \ge 0$.
\end{lemma}

We will use the following result from \cite{MR1404537} on moments of passage times of non-negative processes.
\begin{theorem}[Theorem 1, \cite{MR1404537}]
  \label{thm:AIM}
  Let~$A>0$.
  Suppose we are given an~$(\Fc_n)$-adapted stochastic process~$(X_n)_{n \ge 0}$ taking values in an unbounded subset of~$\R_{\ge 0}$.
  Assume that there exist~$\lambda>0$, $p_0>0$ such that for any~$n$, $X_n^{2p_0}$ is integrable and 
\begin{equation}
\label{eq:martingale-condition}
\E \big(  X^{2p_0}_{n+1} - X^{2p_0}_n \mid \Fc_n \big) \le - \lambda X^{2p_0 - 2}_n \quad \text{on $\{ \tau_A > n \}$}.
\end{equation}
Then, for any positive~$p<p_0$, there exists a positive constant~$\tilde{c}=\tilde{c}(\lambda,p,p_0)$ such that for all~$x \ge 0$, 
\begin{equation*}
\E^x\tau_A^p \le \tilde{c}x^{2p_0}, 
\end{equation*}
where~$\E^x(\cdot) = \E (\cdot \mid X_0 = x)$.
\end{theorem}

\begin{proof}[Proof of Lemma~\ref{lem:passage-time-general-MC-1}]
  Let us fix~$p_0 < \frac{1}{2}$.
  We will show that~\eqref{eq:martingale-condition} can be satisfied for some~$ A > 0$ and~$\lambda = \lambda(A)>0$.

By Taylor expansion, there exists~$\delta>0$ such that 
\begin{equation*}
(1+s)^{2p_0} - 1 \le (2p_0)s + \frac{p_0(2p_0-1)}{2} s^2, \quad |s| \le \delta.
\end{equation*}
Writing~$X_{n+1} = X_n + Z_n$ where~$Z_n$ is independent of~$X_n$, we have
\begin{equation}
\label{eq:40}
\begin{split}
  \E \big( X^{2p_0}_{n+1} - X^{2p_0}_n \mid \Fc_n \big)
&  \le \frac{p_0(2p_0-1)}{2} X_n^{2p_0 -2} \E \big( Z^2_n \ONE_{\{  |Z_n| \le \delta X_n \}} \mid X_n\big)\\
&\   + 2p_0 X_n^{2p_0-1}  \E \big( Z_n \ONE_{\{ |Z_n| \le \delta X_n \}} \mid X_n \big)\\
&\   + \E\big( [  (X_n + Z_n)^{2p_0 } - X_n^{2p_0} ] \ONE_{\{  |Z_n| \ge \delta X_n \}}\mid X_n   \big).
\end{split}
\end{equation}

We can write~$Z_n = Y_1 - Y_2$ where~$(Y_1, Y_2) \sim Q_{0,X_n}$.
When~$X_n \ge A$, the maximal correlation inequality \cite[Lemma 1.1]{Ibragimov} implies that 
\begin{equation*}
|\E Y_1 Y_2 | \le 2\phi(A)^{\frac{1}{2}} \E Y_1^2 = 2\phi(A)^{\frac{1}{2}} m_2, 
\end{equation*}
where~$m_k= \int_{\R} |x|^k\mu(dx)$.
Hence, on $\{ X_n \ge A \}$ we have 
\begin{equation*}
2(1-\phi(A)^{1/2})m_2  \le \E (Z_n^2 \mid X_n).
\end{equation*}
By Markov equality we also have 
\begin{equation}\label{eq:42}
  \Pp \big(| Z_n| \ge \delta X_n \mid X_n \big)
  \le 2 \Pp \big(| Y_1| \ge \frac{\delta}{2} X_n \mid X_n \big) \le
\frac{16m_3}{\delta^3 X_n^3},
\end{equation}
and hence H\"older inequality, 
\begin{equation}
  \label{eq:41}
  \E \big( Y^2_1 \ONE_{\{  |Z_n| > \delta X_n \}} \mid X_n\big)
  \le  \frac{2^{4/3} m_3 }{\delta X_n}.
\end{equation}
Therefore, on~$\{ X_n \ge A \}$ we have
\begin{equation*}
\begin{split}
  \E \big( Z^2_n \ONE_{\{  |Z_n| \le \delta X_n \}} \mid X_n\big)
  &  = \E \big(Z_n^2\mid X_n \big) -  \E \big( Z^2_n \ONE_{\{  |Z_n| > \delta X_n \}} \mid X_n\big)
  \\
&  \ge 2(1-\phi(A)^{1/2}) m_2 -  4\E \big( Y_1^2\ONE_{\{  |Z_n| > \delta X_n \}} \mid X_n\big)
\\
& \ge  2(1-\phi(A)^{1/2}) m_2 -  
\frac{2^{10/3} m_3 }{\delta A}.
\end{split}
\end{equation*}
This shows that the first term on the right hand of~\eqref{eq:martingale-condition} is bounded from above by~$-\lambda X^{2p_0-2}_n$ for some~$\lambda > 0$, if~$A$ is chosen sufficiently large.

To obtain~\eqref{eq:martingale-condition}, it remains to show that the second and third terms in~\eqref{eq:martingale-condition} can be bounded by~$\lambda'X_n^{2p_0-2}$ for arbitrary small~$\lambda'$.
Similarly to~\eqref{eq:41}, we have 
\begin{equation*}
  \big|\E \big( Z_n \ONE_{\{ |Z_n | \ge \delta X_n \}}\mid X_n \big) \big|
  \le 2 \E \big( |Y_1| \ONE_{\{ |Z_n | \ge \delta X_n \}}\mid X_n \big)
  \le \frac{C}{X_n^2}
\end{equation*}
for some constant~$C$ depending on~$\mu$ and~$\delta$.
Since~$\E (Z_n\mid X_n) = 0$, we can bound the second term by using
\begin{equation*}
  |\E \big( Z_n \ONE_{\{ |Z_n| \le \delta X_n \}} \mid X_n \big)|
  =   |\E \big( Z_n \ONE_{\{ |Z_n| > \delta X_n \}} \mid X_n \big)|
  \le \frac{C}{A}\frac{1}{X_n}.
\end{equation*}
Since~$(1+s)^{2p_0} - 1 \le 2p_0 s$ for~$s \ge -1$, we can bound the third term by
\begin{equation*}
  2p_0X_n^{2p_0-1} \E \big(  Z_n \ONE_{\{ |Z_n| \ge \delta X_n \}} \big)
  \le \frac{2p_0 C}{A} X_n^{2p_0-2}.
\end{equation*}
So~\eqref{eq:martingale-condition} indeed holds for sufficiently large~$A$.

Let~$p<p_0$.
To prove~\eqref{eq:passage-time-not-large-MC}, it suffices to establish
\begin{equation}\label{eq:43}
\E^x \tau_0^p \le Cx^{2p_0}, \quad x \ge 1.
\end{equation}

Let~$\rho$ and~$l$ be the constants
corresponding to~$A$ in assumption~\ref{it:minimum-gap}.
Let us  define
\begin{equation*}
  T_0 = \tau_A, \quad T_i = T_{i-1} + \sigma_i, \ i \ge 1
\end{equation*}
where $\sigma_i = \min \{ k\ge l: X_{T_i+k} \le A \}$.
Let~$N  = \min \{  i \ge 0: X_{T_i} = 0 \}$. Then, we have~$\tau_0 \le T_N$
and\begin{equation*}
  \E^x \tau_0^p \le \E^x \big( \tau_A^p + \sum_{i=1}^N \sigma_i^p \big).
\end{equation*}
By Theorem~\ref{thm:AIM} we have~$\E^x \tau_A^p \le \tilde{c}x^{2p_0}$ and
\begin{equation}\label{eq:44}
  \E \sigma_i^p \le l^p +  \int \tilde{c} y^{2p_0}\Pp \big( X_{T_i+l} \in dy \big), \quad i \ge 1.
\end{equation}
Since~$T_i$ are stopping times and~$X_{T_i} \in [0, T]$, the monotonicity of~$\Psi_{i+1,i}(\cdot)$ implies that~$X_{T_i+l}$ is stochastically dominated by~$\tilde{X}_l$, where~$\tilde{X}_n$ is the Markov
chain with the same transition probability as~$X_n$ but has initial condition~$\tilde{X}_0 = A$.
It is easy to see that~$\tilde{X}_l$ has~$(2p_0)$-moment,
so 
\begin{equation*}
l^p +  \int \tilde{c} y^{2p_0}\Pp \big( X_{T_i+l} \in dy \big)
  \le l^p +  \int  \tilde{c}y^{2p_0} \Pp \big( \tilde{X}_{l} \in dy \big) := M
\end{equation*}
gives a uniform upper bound for the right hand side of~\eqref{eq:44}.
By assumption~\ref{it:minimum-gap}, $N$ is dominated by a geometric random variable and hence~$\E N \le \rho^{-1}$.
Therefore, 
\begin{equation*}
\E^x \tau_0^p \le \tilde{c} x^{2p_0} + \E \sum_{i=1}^NM \le  \tilde{c} x^{2p_0} + \frac{M}{\rho}.
\end{equation*}
This proves~\eqref{eq:43}.
\end{proof}

The last lemma is on the extreme values of our random walks.
\begin{lemma}
\label{lem:displacement-of-random-walk}
Let~$S_n$ be a random walk with step distribution~$\mu$ satisfying~\ref{it:first-second-moment} and~\ref{it:higher-moments}.
Let~$S^{*}_n = \max\limits_{0\le k\le n} |S_k|$.  Then for any $p <\frac{1}{2}$, there exists a constant~$C=C_p>0$ such that for~$M\ge \sqrt{n}$, 
\begin{equation}\label{eq:passage-time-not-small-MC}
\Pp (S^{*}_n \ge M)  \le C \Big(  \frac{\sqrt{n}}{M}e^{-\frac{M^2}{18n}} + \frac{1}{n^p} \Big).
\end{equation}
\end{lemma}
\begin{proof}
  By Lemma~\ref{lem:cor-of-KMT}, there is a Brownian motion~$B_t$ such that 
\begin{equation}\label{eq:45}
\Pp \big( \max_{0 \le k \le n} |S_k - B_k| > n^{\gamma} \big) \le \frac{C}{n^p},
\end{equation}
for~$\gamma =\frac{p+1}{\alpha} \in (\alpha^{-1}, \frac{1}{2})$.
Let~$B^{*}_t$ be the maximal process of~$B_t$.
Then 
\begin{align*}
  \Pp (S^{*}_n \ge M)
&  \le \Pp \big( \max_{0 \le k \le n}|S_k - B_k| > n^{\gamma}   \big)\\
&\   + \Pp \big( \max_{0 \le k \le n-1}\sup_{t \in [k,k+1]} |B_t-B_k| \ge \frac{M}{3} \big)
    +\Pp \big( B^{*}_n \ge \frac{M}{3} \big).
\end{align*}
The first term can be estimated by~\eqref{eq:45}, and the last two terms by~\eqref{eq:passage-time-not-small}.
The proof is complete.
\end{proof}

\bibliographystyle{alpha}
\bibliography{CBM}

\begin{thebibliography}{KMT76}

\bibitem[AIM96]{MR1404537}
S.~Aspandiiarov, R.~Iasnogorodski, and M.~Menshikov.
\newblock Passage-time moments for nonnegative stochastic processes and an
  application to reflected random walks in a quadrant.
\newblock {\em Ann. Probab.}, 24(2):932--960, 1996.

\bibitem[Arr79]{ArratiaCBM}
Richard~Alejandro Arratia.
\newblock {\em C{OALESCING} {BROWNIAN} {MOTIONS} {ON} {THE} {LINE}}.
\newblock ProQuest LLC, Ann Arbor, MI, 1979.
\newblock Thesis (Ph.D.)--The University of Wisconsin - Madison.

\bibitem[Bry82]{BryAPPROXIMATIONTHEOREMBERKES1982}
Wlodzimierz Bryc.
\newblock {{ON THE APPROXIMATION THEOREM OF I}}. {{BERKES AND W}}. {{PHILIPP}}.
\newblock {\em Demonstratio Mathematica}, 15(3):807--816, July 1982.

\bibitem[Ibr62]{Ibragimov}
I.~A. Ibragimov.
\newblock Some limit theorems for stationary processes.
\newblock {\em Theory of Probability \& Its Applications}, 7(4):349--382, 1962.

\bibitem[KMT76]{KMT2}
J.~Koml\'{o}s, P.~Major, and G.~Tusn\'{a}dy.
\newblock An approximation of partial sums of independent {RV}'s, and the
  sample {DF}. {II}.
\newblock {\em Z. Wahrscheinlichkeitstheorie und Verw. Gebiete}, 34(1):33--58,
  1976.

\bibitem[NT15]{NTWeakConvergenceLocalized2015}
James Norris and Amanda Turner.
\newblock Weak convergence of the localized disturbance flow to the coalescing
  {{Brownian}} flow.
\newblock {\em The Annals of Probability}, 43(3):935--970, May 2015.

\bibitem[Pit97]{piterbarg1997expansions}
Vladimir~V Piterbarg.
\newblock {\em Expansions and contractions of stochastic flows}.
\newblock University of Southern California, 1997.

\bibitem[PP99]{PPIntermittencyTracerGradient1999a}
Leonid~I. Piterbarg and Vladimir~V. Piterbarg.
\newblock Intermittency of the {{Tracer Gradient}}.
\newblock {\em Communications in Mathematical Physics}, 202(1):237--253, April
  1999.

\end{thebibliography}

\end{document}